\documentclass[12pt]{amsart}
\usepackage{amsmath}
\usepackage{amsfonts}
\usepackage{amssymb}
\usepackage[all,cmtip]{xy}           

\usepackage{bbm}
\usepackage{bbding}
\usepackage{txfonts}
\usepackage[shortlabels]{enumitem}
\usepackage{ifpdf}
\ifpdf
\usepackage[colorlinks,final,backref=page,hyperindex]{hyperref}
\else
\usepackage[colorlinks,final,backref=page,hyperindex,hypertex]{hyperref}
\fi
\usepackage{tikz-cd}
\usepackage[active]{srcltx}
\usepackage{tikz}
\usepackage{amscd}
\usepackage{bm}

\makeatletter

\newtheorem{thm}{Theorem}[section]
\newtheorem{prop}[thm]{Proposition}

\newtheorem{prop-def}{thm}[section]

\theoremstyle{definition}
\newtheorem{defn}[thm]{Definition}

\newtheorem{exam}[thm]{Example}

\newcommand{\nc}{\newcommand}

 \nc{\mbibitem}[1]{\bibitem{#1}} 
 \nc{\mrm}[1]{{\rm #1}}

\nc{\pl}{\cdot}
 \nc{\la}{\longrightarrow}
\nc{\ot}{\otimes}
 \nc{\rar}{\rightarrow}

\nc{\PLA}{{\mathrm{PLA}}}
\nc{\bfk}{{\bf k}}
\nc{\C}{{\mathrm{C}}}
\nc{\RBA}{\mathsf{RBPLA_{\lambda}}}
\nc{\RBO}{\mathsf{RBO}}
\nc{\End}{\mrm{End}}
\nc{\Ext}{\mrm{Ext}}
\nc{\Fil}{\mrm{Fil}}
\nc{\Fr}{\mrm{Fr}}
\nc{\Frob}{\mrm{Frob}}
\nc{\Gal}{\mrm{Gal}}
\nc{\GL}{\mrm{GL}}
\nc{\Hom}{\mrm{Hom}}
\nc{\Hoch}{\mrm{Hoch}}
\nc{\hsr}{\mrm{H}}
\nc{\hpol}{\mrm{HP}}
\nc{\im}{\mrm{im}}
\nc{\Id}{\mrm{Id}}
\nc{\Irr}{\mrm{Irr}}
\nc{\incl}{\mrm{incl}}
\nc{\length}{\mrm{length}}
\nc{\NLSW}{\mrm{NLSW}}
\nc{\Lie}{\mrm{Lie}}
\nc{\Alg}{\mrm{Alg}}
\nc{\mchar}{\rm char}
\nc{\mpart}{\mrm{part}}
\nc{\ql}{{\QQ_\ell}}
\nc{\qp}{{\QQ_p}}
\nc{\rank}{\mrm{rank}}
\nc{\rcot}{\mrm{cot}}
\nc{\rdef}{\mrm{def}}
\nc{\rdiv}{{\rm div}}
\nc{\rmH}{ {\mathrm{H}}}
\nc{\rtf}{{\rm tf}}
\nc{\rtor}{{\rm tor}}
\nc{\res}{\mrm{res}}
\nc{\Sh}{{\mathrm{Sh}}}
\nc{\SL}{\mrm{SL}}
\nc{\Spec}{\mrm{Spec}}
\nc{\sgn}{{\mathrm{sgn}}}
\nc{\tor}{\mrm{tor}}
\nc{\Tr}{\mrm{Tr}}
\nc{\tr}{\mrm{tr}}
\nc{\wt}{\mrm{wt}}
\nc{\op}{\mrm{op}}

\nc{\BA}{{\mathbb A}}   \nc{\CC}{{\mathbb C}}
\nc{\DD}{{\mathbb D}}   \nc{\EE}{{\mathbb E}}
\nc{\FF}{{\mathbb F}}   \nc{\GG}{{\mathbb G}}
\nc{\HH}{ \mathrm{H}}   \nc{\LL}{{\mathbb L}}
\nc{\NN}{{\mathbb N}}   \nc{\PP}{{\mathbb P}}
\nc{\QQ}{{\mathbb Q}}   \nc{\RR}{{\mathbb R}}
\nc{\TT}{{\mathbb T}}   \nc{\VV}{{\mathbb V}}
\nc{\ZZ}{{\mathbb Z}}   \nc{\TP}{\widetilde{P}}
 \nc{\MN}{{\mathrm{MN}}}

\nc{\cala}{{\mathcal A}}    \nc{\calc}{{\mathcal C}}
\nc{\cald}{\mathcal{D}}     \nc{\cale}{{\mathcal E}}
\nc{\calf}{{\mathcal F}}    \nc{\calg}{{\mathcal G}}
\nc{\calh}{{\mathcal H}}    \nc{\cali}{{\mathcal I}}
\nc{\call}{{\mathcal L}}    \nc{\calm}{{\mathcal M}}
\nc{\caln}{{\mathcal N}}    \nc{\calo}{{\mathcal O}}
\nc{\calp}{{\mathcal P}}    \nc{\calr}{{\mathcal R}}
\nc{\cals}{{\mathcal S}}    \nc{\calt}{{\Omega}}
\nc{\calv}{{\mathcal V}}    \nc{\calw}{{\mathcal W}}
\nc{\calx}{{\mathcal X}}

\nc{\fraka}{{\mathfrak a}}
\nc{\frakb}{\mathfrak{b}}
\nc{\frakg}{{\frak g}}
\nc{\frakl}{{\frak l}}
\nc{\fraks}{{\frak s}}

\nc{\frakm}{{\frak m}}
\nc{\frakM}{{\frak M}}
\nc{\frakp}{{\frak p}}
\nc{\frakW}{{\frak W}}
\nc{\frakX}{{\frak X}}
\nc{\frakS}{{\frak S}}
\nc{\frakA}{{\frak A}}
\nc{\frakC}{{\frak{C}}}
\nc{\frakx}{{\frakx}}

\begin{document}

\title[Rota-Baxter pre-Lie algebras]{Cohomology theory of  Rota-Baxter pre-Lie algebras of arbitrary weights}

\author{Shuangjian Guo, Yufei Qin, Kai Wang and Guodong Zhou}
\address{Shuangjian Guo, School of Mathematics and Statistics, Guizhou University of Finance and Economics, Guiyang 550025, China}
\email{shuangjianguo@126.com }

\address{Yufei Qin, Kai Wang and Guodong Zhou,   School of Mathematical Sciences, Shanghai Key laboratory of PMMP,
  East China Normal University,
 Shanghai 200241,
   China}
   \email{290673049@qq.com}
   \email{wangkai@math.ecnu.edu.cn }
\email{gdzhou@math.ecnu.edu.cn}

\date{\today}

\begin{abstract}
This paper is devoted to study deformation, cohomology theory of  Rota-Baxter pre-Lie algebras
of arbitrary weights. First we give the notion of a  new representation of a  Rota-Baxter pre-Lie
algebra of arbitrary weight and define the cohomology theory of a Rota-Baxter pre-Lie
algebra of arbitrary weight.  Then  we study formal deformations by  lower degree cohomology groups.  Finally,   we classify abelian extensions of Rota-Baxter pre-Lie
algebras of arbitrary weight using the
second cohomology group, and classify skeletal Rota-Baxter pre-Lie 2-algebra of arbitrary weight using the third
cohomology group as applications.
\end{abstract}

\subjclass[2010]{
16E40   
16S80   
12H05   
12H10   
16W25   
16S70   
}

\keywords{cohomology,  deformation, pre-Lie algebra,  Rota-Baxter operator, Rota-Baxter pre-Lie 2-algebra}

\maketitle

\tableofcontents

\allowdisplaybreaks

\section*{Introduction} Pre-Lie algebras   (bearing also the names:  left-symmetric algebra \cite{Vinberg}, quasi-associative algebra \cite{Kup},  Vinberg algebra or Koszul algebra or Koszul-Vinberg algebra after   \cite{Kos61, Vinberg}) were first introduced by Cayley \cite{Ca90} in the context of rooted tree algebras. Nowadays, Pre-Lie algebras
have numerous applications and connections to many mathematical branches,
such as affine manifolds, affine structures on
Lie groups and convex homogeneous cones \cite{K86},  complex and symplectic structures on
Lie groups and Lie algebras \cite{AS05}, integrable systems \cite{B90}, classical and quantum
Yang-Baxter equations \cite{ES99,E99}, Poisson brackets and infinite-dimensional Lie
algebras \cite{BN85}, quantum field theory \cite{CK98}, operads \cite{CL01} and geometry and physics \cite{B06}.  See \cite{BaiSurvey} and \cite{Manchon} for more details.

Rota-Baxter algebras were first introduced by Baxter in his study of the fluctuation theory in probability \cite{B60}. Baxter's work was further investigated  by Rota \cite{Ro69} and Cartier \cite{C72}.  Later,   Guo et al. \cite{G00,Gu00, GK00} have done many excellent results as regards concerning Rota-Baxter algebras. In order to develop the deformation theory \cite{Ge63, Ge64} of Rota-Baxter operators.  Tang, Bai, Guo and Sheng \cite{TBGS19} studied  deformation theory
and cohomology theory of $\mathcal{O}$-operators (also called relative Rota-Baxter operators) on Lie algebras,
with applications to Rota-Baxter Lie algebras in mind. Das \cite{D20} developed a similar theory
for Rota-Baxter associative algebras of weight zero.    Recently,  Wang and Zhou \cite{WZ21}, Das \cite{D21} studied Rota-Baxter algebras of arbitrary  weight by different methods  respectively. Later, Das \cite{Das2}  studied cohomology  of weighted Rota-Baxter Lie algebras and Rota-Baxter
paired operators using the method of \cite{WZ21}.    This may be a  transformation from weight zero to arbitrary weight, and the successful  transformation is entirely due to Rota-Baxter Lie groups \cite{GL21}.  Later, Lang and Sheng \cite{LS21}  introduced the notion of a quadratic Rota-Baxter Lie algebra of arbitrary weight , and showed that factorizable Lie bialgebras one-to-one correspond
to quadratic Rota-Baxter Lie algebras of arbitrary weight, Furthermore,  they introduced the notions of matched pairs, bialgebras and Manin triples of Rota-Baxter Lie algebras of arbitrary weight, and showed that Rota-Baxter Lie bialgebras, Manin triples of Rota-Baxter Lie algebras and certain matched pairs of Rota-Baxter Lie algebras are equivalent.

  The notion of a Rota-Baxter operator on a pre-Lie algebra was introduced by Li,  Hou and Bai~\cite{LHB07} for the study of  bialgebra theory of pre-Lie algebras \cite{B08} and $L$-dendriform algebras \cite{BLN10}. Recently,   Liu \cite{L20} studied (quasi-)twilled pre-Lie algebras and the associated $L_{\infty}$-algebras and differential
graded Lie algebras.  Later,    Liu and Wang \cite{LW20} studied pre-Lie analogues of Poisson-Nijenhuis structures and introduced
ON-structures on bimodules over pre-Lie algebras.    There are some other related work \cite{TBGS19, WSBL}. These work all concern Rota-Baxter operators of weight zero.

Up to now, there was very few study about the Rota-Baxter operators of weight zero on pre-Lie algebras.  Not to mention studying  about the Rota-Baxter operators of arbitrary weight on pre-Lie algebras.   Thus it is time to study  deformation, cohomology theory of  Rota-Baxter pre-Lie algebras of arbitrary weights.

In a forthcoming paper, we will study Rota-Baxter pre-Lie algebras from an operadic viewpoint, say, we will exhibit the minimal model of the operad governing Rota-Baxter pre-Lie algebras and deduces from the minimal model the $L_\infty$-structure on the cochain complex of a Rota-Baxter pre-Lie algebra and the notion of homotopy Rota-Baxter pre-Lie algebras. We will also compute the cohomology of some examples.
This paper can be considered as the elementary part of our projet on Rota-Baxter pre-Lie algebras.

The paper is organized as follows. In Section 1, we recall some preliminaries. In Section 2,
we consider  Rota-Baxter pre-Lie algebras of arbitrary weight and introduce their representations. We also provide various examples and new constructions.  In Section 3,  we  define a cohomology theory for Rota-Baxter pre-Lie algebras of arbitrary weight. In Section 4,  we justify this cohomology theory by interpreting lower degree cohomology groups as formal deformations. In Section 5, we study  abelian extensions of
Rota-Baxter pre-Lie algebras of arbitrary weight. In Section 6, we classify skeletal Rota-Baxter pre-Lie 2-algebra of arbitrary weight using the third
cohomology group as applications. In the last section, inspired by \cite{AB08}, we construct a chain map from the cochain complex of a Rota-Baxter associative algebra \cite{WZ21} to that of the corresponding Rota-Baxter pre-Lie algebra.

Throughout this paper, let $\bfk$ be a field of characteristic $0$.  Except specially stated,  vector spaces are  $\bfk$-vector spaces and  all    tensor products are taken over $\bfk$.

\bigskip

\section{Pre-Lie algebras, bimodules and their cohomology theory}
We start with the background of pre-Lie algebras and their cohomology that we refer the reader to the surveys~\cite{B06, Manchon, BaiSurvey} for more details.


\begin{defn}
A {\bf pre-Lie algebra} \cite{Ge63}  is a vector space $A $  together with a binary  operation
$\cdot: A \otimes A  \rightarrow A $ such that, for
 $x, y,z \in A $, the associator $$(x, y, z) = (x\cdot y)\cdot z-x\cdot(y\cdot z)$$
is symmetric in $x, y,$ i.e.
\begin{align*}
(x, y, z) = (y, x, z), \text{ or \, equivalently}, (x \cdot y) \cdot z - x \cdot (y \cdot z) = (y \cdot x) \cdot z - y \cdot (x \cdot z).
\end{align*}
\end{defn}

Let $(A , \cdot)$ be a pre-Lie algebra. Define a new operation on $A $ by setting
\begin{align*}
[x, y]_{A } = x \cdot y - y \cdot x, \text{ for } x, y \in A .
\end{align*}
Then $(A, [,]_{A})$ is a Lie algebra, which is called the {\bf sub-adjacent Lie algebra} of $(A, \cdot)$ and we denote it by $\mathfrak{g}_A$.
   Furthermore, the map $L: A \rightarrow gl(A), \ x \mapsto L_x$, where $L_xy = x \cdot y$  for all $x, y \in A$, gives a  representation  of the Lie algebra $\mathfrak{g}_A$ of $A$.


\begin{defn} \cite{Nij69, Nij70}\cite[Definition 3.1]{B08}
Let $(A, \cdot)$ be a pre-Lie algebra and $M$ a vector space. Let $S: A\otimes M\rightarrow M$ and $P: M \otimes A\rightarrow M$ be
two linear maps.
 Then the triple $(M,  S, P)$ is called a {\bf representation}
over $A$ if
\begin{align*}
S(x\ot S(y\ot u)-S((x\cdot y)\ot u)=& S(y\ot S(x\ot u)-S((y\cdot x)\ot u)\\
S(x\ot P(u\ot y))-P(S(x\ot u)\ot y)=&P(u\ot  (x\cdot y))-P(P(u\ot x) \ot y),
\end{align*}
for all  $x, y \in A, u \in M$.

 By abuse of notations, we shall write $x\cdot u=S(x\ot u)$ and $u\cdot y=P(u\ot y)$.
\end{defn}

 Let us recall the cohomology theory of pre-Lie algebras following \cite{Kos68, Nij69, Nij70, Dzh}; for homology of pre-Lie algebras, we refer the reader to \cite{Nij69,Nij70,  CL01}. Let $(A, \cdot)$ be a pre-Lie algebra and  $(M, S, P)$ be a representation of it. The  {\bf cohomology} of $A$ with coefficients in $M$ is the
cohomology of the cochain complex $$\{ \C^\bullet_\PLA(A, M), \partial_\PLA^\bullet \},$$ where $$\C^n (A, M) =  \Hom   ( \wedge^{n-1} A\ot A, M ), n \geq 1$$ and the
coboundary operator $$\partial_\PLA^n: \C^n(A, M ) \rightarrow \C^{n+1}(A, M), n\geq 1$$ is given by
\begin{eqnarray*}
&&(\partial_\PLA^n f)(x_1, ... , x_{n+1})\\
&=&\sum^{n}_{i=1}(-1)^{i+1}x_i\cdot f(x_1, ... , \hat{x}_i, ... , x_{n+1})+ \sum^{n}_{i=1}(-1)^{i+1}f(x_1, ... ,\hat{x}_i, ... ,  x_{n}, x_i)\cdot x_{n+1}\\
&&-\sum^{n}_{i=1}(-1)^{i+1}f(x_1,..., \hat{x}_i, ... , x_n,  x_i\cdot x_{n+1})+\sum_{1\leq i< j\leq n} (-1)^{i+j}f([x_i, x_j]_A, x_1, ... , \hat{x}_i, ... , \hat{x}_{j}, ... , x_{n+1}),
\end{eqnarray*}
for $x_1, ..., x_{n+1}\in A$. The corresponding cohomology groups are denoted by $\HH^*_\PLA(A, M).$

 \bigskip

\section{Rota-Baxter operators on pre-Lie algebras}

In  this section, we introduce basic definitions and facts about  Rota-Baxter pre-Lie algebra.

\begin{defn}\label{Def: Rota-Baxter pre-Lie algebra} Let $\lambda\in \bfk$.
Given  a pre-Lie algebra  $(\frakg,  \mu=\pl)$, it is called  a Rota-Baxter pre-Lie algebra  of weight $\lambda$, if there is a linear
operator $T: \frakg \rightarrow \frakg $ subject to
\begin{eqnarray}\label{Eq: Rota-Baxter relation}
T(a)\pl  T(b)=T\big(a\pl  T(b)+T(a)\pl  b+\lambda\  a\pl
b\big)\end{eqnarray}	
for arbitrary $a,~b \in \frakg $, or in terms of maps
\begin{eqnarray}\label{Eq: Rota-Baxter relation in terms of maps}
\mu\circ (T \ot T)=T\circ  \mu\circ (\Id\ot T +T \ot \Id)+\lambda\ T\circ \mu.\end{eqnarray} In this case,  $T$ is called a Rota-Baxter operator of weight $\lambda$.

Denote by $\RBA$ the category of Rota-Baxter pre-Lie algebras with obvious morphisms.
\end{defn}

Inspired by \cite{Das2},  we have the following example.
\begin{exam}
 Let $(\frakg, \cdot,  T)$ be a  Rota-Baxter pre-Lie algebra  of weight $\lambda$.
\begin{itemize}
\item[(i)]   For arbitrary $\lambda'\in \bfk$, the triple $(\frakg, \cdot,  \lambda' T)$ is a Rota-Baxter pre-Lie algebra  of weight $\lambda\cdot\lambda'$.

\item[(ii)]    The triple  $(\frakg, \cdot,  -\lambda  \Id_\frakg-T)$ is a Rota-Baxter pre-Lie algebra  of weight $\lambda$.

\item[(iii)]   Given an automorphism $\psi\in$ Aut$(\frakg)$ of the pre-Lie
algebra $\frakg$, the triple $(\frakg, \cdot, \psi^{-1} \circ T \circ \psi)$ is a Rota-Baxter pre-Lie algebra  of weight $\lambda$.
\end{itemize}
\end{exam}

\begin{exam}
 Let  $(\mathfrak{g}, \cdot)$ be a 3-dimensional pre-Lie algebra and $\{ e_1, e_2, e_3\}$ be a basis,   whose nonzero products are given as follows:
\begin{eqnarray*}
e_3 \cdot e_3=e_2.
\end{eqnarray*}
Then  $T=\left(
\begin{array}{ccc}
 a_{11} & a_{12}  & a_{13}  \\
  a_{21} & a_{22} & a_{23} \\
   a_{31} & a_{32} & a_{33} \\
   \end{array}
   \right)$
is a Rota-Baxter operator of weight $\lambda$   if and only if
\begin{eqnarray*}
Te_i \cdot Te_j = T (T e_i \cdot e_j+ e_i \cdot Te_j +\lambda e_i \cdot e_j), ~ \text{ for } i, j=1, 2, 3.
\end{eqnarray*}
We have
$$T e_1 \cdot Te_1=(a_{11}e_1 + a_{21}e_2 + a_{31}e_3)\cdot( a_{11}e_1 + a_{21}e_2 + a_{31}e_3) = a^2_{31}e_2.$$
Moreover,
\begin{eqnarray*}
T(T e_1\cdot e_1 +e_1\cdot Te_1+\lambda e_1 \cdot e_1)=0.
\end{eqnarray*}
Then  it follows from $T e_1 \cdot Te_1=T (T e_1\cdot e_1 +e_1\cdot Te_1+\lambda e_1 \cdot e_1)$  that
\begin{eqnarray*}
a_{31}=0.
\end{eqnarray*}
By considering other choices of $e_i$ and $e_j$ except $e_3$ and $e_3$, we obtain
\begin{eqnarray*}
a_{32}=0.
\end{eqnarray*}
Finally, we consider $e_3$ and $e_3$, we obtain
$$T e_3 \cdot Te_3=(a_{13}e_1 + a_{23}e_2 + a_{33}e_3)\cdot( a_{13}e_1 + a_{23}e_2 + a_{33}e_3) = a^2_{33}e_2.$$
Moreover,
\begin{eqnarray*}
&&T(T e_3\cdot e_3 +e_3\cdot Te_1+\lambda e_3 \cdot e_3)\\
&=& (2a_{33}+\lambda) T(e_2)\\
&=& (2a_{33}+\lambda)a_{12} e_1+(2a_{33}+\lambda)a_{22} e_2.
\end{eqnarray*}
Then  it follows from $T e_3 \cdot Te_3=T (T e_3\cdot e_3 +e_3\cdot Te_3+\lambda e_3 \cdot e_3)$  that
\begin{eqnarray*}
(2a_{33}+\lambda)a_{12}=0, (2a_{33}+\lambda)a_{22} = a^2_{33}.
\end{eqnarray*}
Summarize the above discussion.

(1) If $a_{12}=a_{22}=a_{31}=a_{32}=a_{33}=0$,  then any
$T=\left(
\begin{array}{ccc}
a_{11} & 0 & a_{13}  \\
a_{21} & 0 & a_{23} \\
0 & 0 & 0 \\
\end{array}
\right)$  is a  Rota-Baxter operator of weight $\lambda$.

(2) If $a_{12}=a_{31}=a_{32}=0$ and $a_{22}\neq 0$, then any
$T=\left(
\begin{array}{ccc}
a_{11} & 0 & a_{13}  \\
a_{21} & a_{22} & a_{23} \\
0 & 0 & a_{33} \\
\end{array}
\right)$  is a  Rota-Baxter operator of weight $\lambda=\frac{a_{33}^2}{a_{22}}-2a_{33}$.
\end{exam}




\begin{defn}\label{Def: Rota-Baxter bimodules} Let $(\frakg, \mu, T)$ be a Rota-Baxter pre-Lie algebra and $(M, \cdot)$ be a bimodule over pre-Lie
algebra $(\frakg, \mu)$. We say that $(M, \cdot, T_M)$ is a bimodule over Rota-Baxter pre-Lie algebra $(\frakg, \mu, T)$ or a Rota-Baxter
bimodule if $(M, \cdot)$ is endowed with a linear operator $T_M : M \rightarrow M$ such that the following equations
\begin{eqnarray}
&&T(a)\cdot T_M(m) = T_M(a\cdot T_M(m) + T(a)\cdot m + \lambda a\cdot m),\\
&&T_M(m)\cdot T(a) = T_M(m\cdot T(a) + T_M(m)\cdot a + \lambda m\cdot a)
\end{eqnarray}
hold for arbitrary  $a \in \frakg$ and $m \in M$.

\end{defn}
It is ready to see that  $(\frakg ,\mu, T)$ itself is a bimodule over the Rota-Baxter pre-Lie algebra $(\frakg ,\mu,T)$, called the regular Rota-Baxter bimodule.





Recall first the  following interesting observation:
\begin{prop} \label{Prop: new RB pre-Lie algebra}
	Let $(\frakg, \pl,T)$ be a Rota-Baxter pre-Lie algebra. Define a new binary operation as:
\begin{eqnarray}\end{eqnarray}
for arbitrary $a,b\in \frakg $. Then
\begin{itemize}

\item[(i)]      $(\frakg ,\star )$ is a new  pre-Lie algebra;

    \item[(ii)] the triple  $(\frakg ,\star ,T)$ also forms a Rota-Baxter pre-Lie algebra    and denote it by $\frakg _\star $;

\item[(iii)] the map $T:(\frakg ,\star , T)\rightarrow (\frakg,\pl, T)$ is a  morphism of Rota-Baxter pre-Lie  algebras.
\item[(iv)] Define a new operation on $\frakg $ by setting
\begin{align*}
	[x, y]_{\star } = x \star y - y \star x, \text{ for } x, y \in \frakg .
\end{align*}
Then $(\frakg, [\cdot, \cdot]_{\star })$ is a Lie algebra.
    \end{itemize}
	\end{prop}

One can also construct new Rota-Baxter bimodules from old ones.
\begin{prop}\label{Prop: new-bimodule}
Let $(\frakg, \pl,T)$ be a Rota-Baxter pre-Lie algebra  and $(M, \pl, T_M)$ be a Rota-Baxter bimodule over it. We define a left action $``\rhd"$ and a right action $``\lhd"$ of $\frakg $ on $M$ as follows: for arbitrary $a\in \frakg ,m\in M$,
\begin{eqnarray}
a\rhd m:&=& T(a)\pl m-T_M(a \pl m),\\
m\lhd a:&=& m\pl T(a)-T_M(m \pl a).
\end{eqnarray}
Then these actions make $M$ into a Rota-Baxter  bimodule   over $\frakg _\star $ and denote this new  bimodule  by $_\rhd M_\lhd$.
\end{prop}

\begin{proof}
	Firstly, we show that $(M, \rhd, \lhd )$ is a bimodule over $(\frakg, \star )$. In fact, we have
\begin{eqnarray*}
&& a\rhd(b\rhd m)-(a\star b)\rhd m\\
&=& a\rhd(T(b)\pl m-T_M(b \pl m))-T(a\star b)\pl m+T_M((a\star b) \pl m)\\
&=& T(a)\cdot (T(b)\cdot m - T_M(b\cdot m))- T_M(a\cdot (T(b)\cdot m)+ a\cdot T_M(b\cdot m))\\
&&-(T(a)\cdot T(b))\cdot m + T_M((a\cdot T(b))\cdot m + (T(a)\cdot b)\cdot m + \lambda (a\cdot b)\cdot m)\\
&=&T(a)\cdot (T(b)\cdot m) -T(a)\cdot T_M(b\cdot m)- T_M(a\cdot (T(b)\cdot m)- a\cdot T_M(b\cdot m))\\
&&-(T(a)\cdot T(b))\cdot m + T_M((a\cdot T(b))\cdot m + (T(a)\cdot b)\cdot m + \lambda (a\cdot b)\cdot m)\\
&=&T(a)\cdot (T(b)\cdot m) -T_M(a\cdot T_M(b\cdot m) + T(a)\cdot (b\cdot m) + \lambda a\cdot (b\cdot m))\\
&&- T_M(a\cdot (T(b)\cdot m)- a\cdot T_M(b\cdot m))-(T(a)\cdot T(b))\cdot m \\
&&+ T_M((a\cdot T(b))\cdot m + (T(a)\cdot b)\cdot m + \lambda (a\cdot b)\cdot m)\\
&=&T(a)\cdot (T(b)\cdot m) -T_M(a\cdot (T(b)\cdot m) + T(a)\cdot (b\cdot m) + \lambda a\cdot (b\cdot m))\\
&&-(T(a)\cdot T(b))\cdot m + T_M((a\cdot T(b))\cdot m + (T(a)\cdot b)\cdot m + \lambda (a\cdot b)\cdot m).
\end{eqnarray*}
On the other hand, we have
\begin{eqnarray*}
&& b\rhd(a\rhd m)-(b\star a)\rhd m\\
&=&T(b)\cdot (T(a)\cdot m) -T_M(b\cdot (T(a)\cdot m) + T(b)\cdot (a\cdot m) + \lambda b\cdot (a\cdot m))\\
&&-(T(b)\cdot T(a))\cdot m + T_M((b\cdot T(a))\cdot m + (T(b)\cdot a)\cdot m + \lambda (b\cdot a)\cdot m).
\end{eqnarray*}
So we have
\begin{eqnarray*}
&& a\rhd(b\rhd m)-(a\star b)\rhd m=b\rhd(a\rhd m)-(b\star a)\rhd m.
\end{eqnarray*}

Similarly, we have
\begin{eqnarray*}
&& a\rhd (m\lhd b) - (a\rhd m) \lhd b\\
&=& a\rhd (m\pl T(b)-T_M(m \pl b)) - (T(a)\pl m-T_M(a \pl m)) \lhd b\\
&=& T(a)\cdot (m\pl T(b)-T_M(m \pl b))-T_M(a\pl (m\pl T(b)-T_M(m \pl b)))\\
&& -(T(a)\pl m-T_M(a \pl m))\cdot T(b)+T_M((T(a)\pl m-T_M(a \pl m))\cdot b)\\
&=& T(a)\cdot (m\pl T(b))-T(a)\cdot T_M(m \pl b)-T_M(a\pl (m\pl T(b)-T_M(m \pl b)))\\
&& -(T(a)\pl m)\cdot T(b)+T_M(a \pl m)\cdot T(b)+T_M((T(a)\pl m-T_M(a \pl m))\cdot b)\\
&=& T(a)\cdot (m\pl T(b))-T_M(T(a)\cdot (m\cdot b) + \lambda a\cdot (m\cdot b))-T_M(a\pl (m\pl T(b))\\
&&-(T(a)\pl m)\cdot T(b) +T_M((a \cdot m)\cdot T(b) + \lambda (a \cdot m)\cdot b)+T_M((T(a)\pl m)\cdot b).
\end{eqnarray*}
On the other hand, we have
\begin{eqnarray*}
&& m \lhd (a\star b)- (m\lhd a) \lhd b \\
&=& m\pl T(a\star b)-T_M(m \pl (a\star b))-(m\pl T(a)-T_M(m \pl a))\lhd b \\
&=& m\pl (T(a)\cdot T(b))-T_M(m \pl (a\pl  T(b)+T(a)\pl  b+\lambda a\pl  b))\\
&& -(m\pl T(a)-T_M(m \pl a))\pl T(b)+T_M((m\pl T(a)-T_M(m \pl a))\pl b)\\
&=& m\pl (T(a)\cdot T(b))-T_M(m \pl (a\pl  T(b)+T(a)\pl  b+\lambda a\pl  b))\\
&& -(m\pl T(a))\pl T(b) +T_M(m \pl a)\pl T(b)+T_M((m\pl T(a)-T_M(m \pl a))\pl b)\\
&=& m\pl (T(a)\cdot T(b))-T_M(m \pl (a\pl  T(b)+T(a)\pl  b+\lambda a\pl  b))-(m\pl T(a))\pl T(b) \\
&& +T_M((m\pl a)\cdot T(b) + T_M(m\pl a)\cdot b + \lambda (m\pl a)\cdot b)+T_M((m\pl T(a)-T_M(m \pl a))\pl b)\\
&=& m\pl (T(a)\cdot T(b))-T_M(m \pl (a\pl  T(b)+T(a)\pl  b+\lambda a\pl  b))-(m\pl T(a))\pl T(b) \\
&& +T_M((m\pl a)\cdot T(b)  + \lambda (m\pl a)\cdot b)+T_M((m\pl T(a))\pl b).
\end{eqnarray*}
So, we have
\begin{eqnarray*}
&& a\rhd (m\lhd b) - (a\rhd m) \lhd b=m \lhd (a\star b)- (m\lhd a) \lhd b.
\end{eqnarray*}
Finally,  It is easy to check that $M$ is  a Rota-Baxter  bimodule   over $\frakg _\star $ and we omit it.
\end{proof}

\bigskip

\section{Cohomology theory of Rota-Baxter pre-Lie algebras} \label{Sect: Cohomology theory of Rota-Baxter pre-Lie algebras}
In this  section, we will establish a cohomology theory for   Rota-Baxter pre-Lie algebras of arbitrary weight.

\subsection{Cohomology of Rota-Baxter Operators}\
\label{Subsect: cohomology RB operator}

Firstly, let's study the cohomology of  Rota-Baxter
operators.

 Let $(\frakg , \mu, T)$ be a
Rota-Baxter pre-Lie algebra and $(M, \cdot, T_M)$ be a Rota-Baxter bimodule over it. Recall that
Proposition~\ref{Prop: new RB pre-Lie algebra}  and Proposition~\ref{Prop: new-bimodule}  give a new
 pre-Lie algebra   $\frakg _\star $ and
  a new   Rota-Baxter bimodule  $_\rhd M_\lhd$ over $\frakg _\star $.
 Consider the Hochschild cochain complex of $\frakg _\star $ with
 coefficients in $_\rhd M_\lhd$:
 $$\C^\bullet_{\PLA}(\frakg _\star , {_\rhd
 	M_\lhd})=\bigoplus\limits_{n=0}^\infty \C^n_{\PLA}(\frakg _\star , {_\rhd
 	M_\lhd}).$$
  More precisely,  for $n\geqslant 0$,  $ \C^n_{\PLA}(\frakg _\star , {_\rhd M_\lhd})=\Hom  (\wedge^{n-1}\frakg \ot \frakg, M)$ and its differential $$\partial^n:
 \C^n_{\mathrm{PLA}}(\frakg _\star ,\  _\rhd M_\lhd)\rightarrow  \C^{n+1}_{\mathrm{PLA}}(\frakg _\star , {_\rhd M_\lhd}) $$ is defined as:

 \begin{align*}&\partial^n(f)(a_{1, n+1}) \\
 	=&\sum^{n}_{i=1}(-1)^{i+1}x_i\rhd f(x_1, ... , \hat{x}_i, ... , x_{n+1})+ \sum^{n}_{i=1}(-1)^{i+1}f(x_1, ... ,\hat{x}_i, ... ,  x_{n}, x_i)\lhd x_{n+1}\\
 	&-\sum^{n}_{i=1}(-1)^{i+1}f(x_1,..., \hat{x}_i, ... , x_n,  x_i\star x_{n+1})+\sum_{1\leq i< j\leq n} (-1)^{i+j}f([x_i, x_j]_{\star}, x_1, ... , \hat{x}_i, ... , \hat{x}_{j}, ... , x_{n+1})\\
 	=&\sum^{n}_{i=1}(-1)^{i+1}T\left( x_i\right) \cdot f(x_1, ... , \hat{x}_i, ... , x_{n+1})-\sum^{n}_{i=1}(-1)^{i+1}T_{M}\left( x_i\cdot f(x_1, ... , \hat{x}_i, ... , x_{n+1})\right) \\
 	&\sum^{n}_{i=1}(-1)^{i+1}f(x_1, ... ,\hat{x}_i, ... ,  x_{n}, x_i)\cdot T\left( x_{n+1}\right) -\sum^{n}_{i=1}(-1)^{i+1}T_{M}\left( f(x_1, ... ,\hat{x}_i, ... ,  x_{n}, x_i)\cdot x_{n+1}\right) \\
 	&-\sum^{n}_{i=1}(-1)^{i+1}f(x_1,..., \hat{x}_i, ... , x_n,  x_i\cdot T\left(  x_{n+1})\right) -\sum^{n}_{i=1}(-1)^{i+1}f(x_1,..., \hat{x}_i, ... , x_n, T\left(  x_{i}\right) \cdot x_{n+1})\\
 	&-\sum^{n}_{i=1}(-1)^{i+1}f(x_1,..., \hat{x}_i, ... , x_n, \lambda  x_i\cdot x_{n+1})+\sum_{1\leq i< j\leq n} (-1)^{i+j}f([T(x_i), x_j]_{\frakg}, x_1, ... , \hat{x}_i, ... , \hat{x}_{j}, ... , x_{n+1})\\
 	&+\sum_{1\leq i< j\leq n} (-1)^{i+j}f([x_i,T( x_j)]_{\frakg}, x_1, ... , \hat{x}_i, ... , \hat{x}_{j}, ... , x_{n+1})+\sum_{1\leq i< j\leq n} (-1)^{i+j}f(\lambda [x_i, x_j]_{\frakg}, x_1, ... , \hat{x}_i, ... , \hat{x}_{j}, ... , x_{n+1})
\end{align*}
for arbitrary $f\in  \C^n_{\PLA}(A_\star ,\  _\rhd M_\lhd)$ and $a_1,\dots,a_{n+1}\in A$.

 \smallskip

 \begin{defn}
 	Let $\frakg =(\frakg ,\mu,T)$ be a Rota-Baxter pre-Lie algebra of weight $\lambda$ and $M=(M, \cdot, T_M)$ be a Rota-Baxter bimodule over it. Then the cochain complex $(\C^\bullet_\PLA(\frakg _\star, {_\rhd M_\lhd}),\partial)$ is called the cochain complex of Rota-Baxter operator $T$ with coefficients in $(M, T_M)$,  denoted by $C_{\RBO}^\bullet(\frakg , M)$. The cohomology of $C_{\RBO}^\bullet(\frakg ,M)$, denoted by $\mathrm{H}_{\RBO}^\bullet(\frakg ,M)$, are called the cohomology of Rota-Baxter operator $T$ with coefficients in $(M, T_M)$.
 	
 	 When $(M,\cdot, T_M)$ is the regular Rota-Baxter bimodule $ (\frakg , \cdot,   T)$, we denote $\C^\bullet_{\RBO}(\frakg ,\frakg )$ by $\C^\bullet_{\RBO}(\frakg )$ and call it the cochain complex of  Rota-Baxter operator $T$, and denote $\rmH^\bullet_{\RBO}(\frakg ,\frakg )$ by $\rmH^\bullet_{\RBO}(\frakg )$ and call it the cohomology of Rota-Baxter operator $T$.
 \end{defn}

\subsection{Cohomology of Rota-Baxter pre-Lie algebras}\
\label{Subsec:chomology RB}

In this subsection, we will combine the  cohomology of   pre-Lie algebras and the cohomology of Rota-Baxter operators to define a cohomology theory for Rota-Baxter pre-Lie algebras.

Let $M=(M,\cdot, T_M)$ be a  Rota-Baxter bimodule over a Rota-Baxter pre-Lie algebra   $\frakg =(\frakg ,\mu,T)$. Now, let's construct a chain map   $$\Phi^\bullet:\C^\bullet_{\PLA}(\frakg ,M) \rightarrow C_{\RBO}^\bullet(\frakg ,M),$$ i.e., the following commutative diagram:
\[\xymatrix{
		\C^1_{\PLA}(\frakg ,M)\ar[r]^-{\partial_\PLA^1}\ar[d]^-{\Phi^1}& \C^2_{\PLA}(\frakg,M)\ar@{.}[r]\ar[d]^-{\Phi^1}&\C^n_{\PLA}(\frakg,M)\ar[r]^-{\partial_\PLA^n}\ar[d]^-{\Phi^n}&\C^{n+1}_{\PLA}(\frakg,M)\ar[d]^{\Phi^{n+1}}\ar@{.}[r]&\\
		\C^1_{\RBO}(\frakg,M)\ar[r]^-{\partial^1}&\C^2_{\RBO}(\frakg,M)\ar@{.}[r]& \C^n_{\RBO}(\frakg,M)\ar[r]^-{\partial^n}&\C^{n+1}_{\RBO}(\frakg,M)\ar@{.}[r]&
.}\]

 For  $n\geqslant 1$ and $ f\in \C^n_{\PLA}(\frakg,M)$,  define $\Phi^n(f)\in \C^n_{\RBO}(\frakg,M)$ as:
\begin{align*}
  \Phi^n(f) =&f\circ (T^{\wedge(n-1)}\ot T)\\
 &-\sum_{k=1}^{n}\lambda^{n-k}\sum_{1\leqslant i_1<i_2<\cdots< i_{k-1}\leqslant n-1}  T_M\circ f\circ (\Id^{\wedge (i_1-1)} \wedge  T \wedge \Id^{\wedge (i_2-i_1-1)} \wedge T\wedge \cdots\wedge T\wedge \Id^{\wedge (n-1-i_{k-1})}\ot \Id  ) \\
 & -\sum_{k=2}^{n}\lambda^{n-k}\sum_{1\leqslant i_1< \cdots< i_{k-2}\leqslant n-1}  T_M\circ f\circ (\Id^{\wedge (i_1-1)} \wedge  T \wedge \Id^{\wedge (i_2-i_1-1)} \wedge T\wedge \cdots\wedge T\wedge \Id^{\wedge (n-1-i_{k-2})}\ot T  ).
\end{align*}

\smallskip

\begin{prop}\label{Prop: Chain map Phi}
	The map $\Phi^\bullet: \C^\bullet_\PLA(\frakg,M)\rightarrow \C^\bullet_{\RBO}(\frakg,M)$ is a chain map.
\end{prop}


  The proof of this result is rather long, but it is a routine verification and, although harder, is similar to the proof of \cite[Proposition 3.3]{GLSZ}. So we omit it. The curious reader could find the complete proof in  \cite{GQWZ22b} which is a preliminary  version of this paper.

\smallskip

\begin{defn}
 Let $M=(M,\cdot, T_M)$ be a  Rota-Baxter bimodule over a Rota-Baxter pre-Lie algebra $\frakg=(\frakg,\mu,T)$  of weight $\lambda$ .  We define the  cochain complex $(\C^\bullet_{\RBA}(\frakg,M), d^\bullet)$  of Rota-Baxter pre-Lie algebra $(\frakg,\mu,T)$ with coefficients in $(M,\cdot, T_M)$ to the negative shift of the mapping cone of $\Phi^\bullet$, that is,   let
\[\C^0_{\RBA}(\frakg,M)=\C^0_\PLA(\frakg,M)  \quad  \mathrm{and}\quad   \C^n_{\RBA}(\frakg,M)=\C^n_\PLA(\frakg,M)\oplus \C^{n-1}_{\RBO}(\frakg,M), \forall n\geqslant 1,\]
 and the differential $d^n: \C^n_{\RBA}(\frakg,M)\rightarrow \C^{n+1}_{\RBA}(\frakg,M)$ is given by \[d^n(f,g)= (\delta^n(f), -\partial^{n-1}(g)  -\Phi^n(f))\]
 for arbitrary $f\in \C^n_\PLA(\frakg,M)$ and $g\in \C^{n-1}_{\RBO}(\frakg,M)$.
The  cohomology of $(\C^\bullet_{\RBA}(\frakg,M), d^\bullet)$, denoted by $\rmH_{\RBA}^\bullet(\frakg,M)$,  is called the cohomology of the Rota-Baxter pre-Lie algebra $(\frakg,\mu,T)$ with coefficients in $(M,\cdot, T_M)$.
When $(M,\cdot, T_M)=(\frakg,\mu, T)$, we just denote $\C^\bullet_{\RBA}(\frakg,\frakg), \rmH^\bullet_{\RBA}(\frakg,\frakg)$   by $\C^\bullet_{\RBA}(\frakg),  \rmH_{\RBA}^\bullet(\frakg)$ respectively, and call  them the cochain complex, the cohomology of Rota-Baxter pre-Lie algebra $(\frakg,\mu,T)$ respectively.
\end{defn}
There is an obvious short exact sequence of complexes:
\begin{eqnarray}\label{Seq of complexes} 0\to s\C^\bullet_{\RBO}(\frakg,M)\to \C^\bullet_{\RBA}(\frakg,M)\to \C^\bullet_{\PLA}(\frakg,M)\to 0\end{eqnarray}
which induces a long exact sequence of cohomology groups
$$0\to \rmH^{0}_{\RBA}(\frakg, M)\to\mathrm{HH}^0(\frakg, M)\to\rmH^0_{\RBO}(\frakg, M) \to \rmH^{1}_{\RBA}(\frakg, M)\to\mathrm{HH}^1(\frakg, M)\to $$
$$\cdots\to \mathrm{HH}^p(\frakg, M)\to \rmH^p_{\RBO}(\frakg, M)\to \rmH^{p+1}_{\RBA}(\frakg, M)\to \mathrm{HH}^{p+1}(\frakg, M)\to \cdots$$


\bigskip


\bigskip

\section{Formal deformations}

In this section, we will study   deformations of Rota-Baxter pre-Lie algebras and interpret  them  via    lower degree   cohomology groups  of Rota-Baxter pre-Lie algebras defined in last section.

\subsection{Formal deformations of Rota-Baxter pre-Lie algebras}\

Let $(\frakg,\mu, T)$ be a Rota-Baxter pre-Lie algebra of weight $\lambda$.   Consider a 1-parameterized family:
\[\mu_t=\sum_{i=0}^\infty \mu_it^i, \ \mu_i\in \C^2_\PLA(\frakg),\quad  T_t=\sum_{i=0}^\infty T_it^i,  \ T_i\in \C^1_{\RBO}(\frakg).\]

\begin{defn}
	A  1-parameter formal deformation of    Rota-Baxter pre-Lie algebra $(\frakg, \mu,T)$ is a pair $(\mu_t,T_t)$ which endows the flat $\bfk[[t]]$-module $\frakg[[t]]$ with a  Rota-Baxter pre-Lie algebra structure over $\bfk[[t]]$ such that $(\mu_0,T_0)=(\mu,T)$.
\end{defn}

 Power series $\mu_t$ and $ T_t$ determine a  1-parameter formal deformation of Rota-Baxter pre-Lie algebra $(\frakg,\mu,T)$ if and only if for arbitrary $a,b,c\in \frakg$, the following equations hold :
 \begin{eqnarray*}
 \mu_t(\mu_t(a\ot b)\ot c)-\mu_t(a\ot \mu_t(b\ot c))&=& \mu_t(\mu_t(b\ot a)\ot c)-\mu_t(b\ot \mu_t(a\ot c)),\\
 \mu_t(T_t(a)\ot T_t(b))&=& T_t\Big(\mu_t(a\ot T_t(b))+\mu_t(T_t(a)\ot b)+\lambda \mu_t(a\ot b)\Big).
 \end{eqnarray*}
By expanding these equations and comparing the coefficient of $t^n$, we obtain  that $\{\mu_i\}_{i\geqslant0}$ and $\{T_i\}_{i\geqslant0}$ have to  satisfy: for arbitrary $n\geqslant 0$,
\begin{equation}\label{Eq: deform eq for  products in RBA}
\sum_{i=0}^n\mu_i\circ(\mu_{n-i}(a\ot b)\ot c)-\sum_{i=0}^n\mu_i\circ(a\ot \mu_{n-i}(b\ot c))
=\sum_{i=0}^n\mu_i\circ(\mu_{n-i}(b\ot a)\ot c)-\sum_{i=0}^n\mu_i\circ(b\ot \mu_{n-i}(a\ot c)),\end{equation}
\begin{equation}\label{Eq: Deform RB operator in RBA} \begin{array}{rcl}
\sum_{i+j+k=n\atop i, j, k\geqslant 0}	\mu_{i}\circ(T_j\ot T_{k})&=&\sum_{i+j+k=n\atop i, j, k\geqslant 0} T_{i}\circ \mu_j\circ (\Id\ot T_{k})\\
&  &+\sum_{i+j+k=n\atop i, j, k\geqslant 0} T_{i}\circ\mu_j\circ (T_{k}\ot \Id)+\lambda\sum_{i+j=n\atop i, j \geqslant 0}T_i\circ\mu_{j}.
\end{array}\end{equation}
Obviously, when $n=0$, the above conditions are exactly the associativity of $\mu=\mu_0$ and Equation~(\ref{Eq: Rota-Baxter relation}) which is the defining relation of Rota-Baxter operator $T=T_0$.

\smallskip

\begin{prop}\label{Prop: Infinitesimal is 2-cocyle}
	Let $(\frakg[[t]],\mu_t,T_t)$ be a  1-parameter formal deformation of
	Rota-Baxter pre-Lie algebra $(\frakg,\mu,T)$ of weight $\lambda$. Then
	$(\mu_1,T_1)$ is a 2-cocycle in the cochain complex
	$C_{\RBA}^\bullet(\frakg)$.
\end{prop}
\begin{proof} When $n=1$,   Equations~(\ref{Eq: deform eq for  products in RBA}) and (\ref{Eq: Deform RB operator in RBA})  become
	 \begin{align*}
&\mu_1(a\ot b)\cdot c+ \mu_1((a\cdot b)\otimes c)-a\cdot \mu_1(b\ot c) -\mu_1(a\otimes b\cdot c) \\
&=\mu_1(b\ot a)\cdot c+ \mu_1((b\cdot a)\otimes c)-b\cdot \mu_1(a\ot c) -\mu_1(b\otimes a\cdot c),
 \end{align*}
and
$$\begin{array}{cl}
	&\mu_1 (T\ot T)-\{T\circ\mu_1\circ(\Id\ot T)+T\circ\mu_1\circ(T\ot \Id)+\lambda T\circ \mu_1\}\\
	=&-\{\mu\circ(T\ot T_1)+\mu\circ(T_1\ot T)-T\circ\mu\circ(\Id\ot T_1)-T\circ\mu\circ(T_1\ot\Id)\}\\
	&+\{T_1\circ\mu\circ(\Id\ot T)+T_1\circ\mu\circ(T\ot \Id)+\lambda T_1\circ \mu\}.
	\end{array}$$
Note that  the first equation is exactly $\delta^2(\mu_1)=0\in \C^\bullet_{\PLA}(\frakg)$ and that  second equation is exactly to  \[\Phi^2(\mu_1)=-\partial^1(T_1) \in \C^\bullet_{\RBO}(\frakg).\]
	So $(\mu_1,T_1)$ is a 2-cocycle in $\C^\bullet_{\RBA}(\frakg)$.
	\end{proof}

\smallskip

\begin{defn} The 2-cocycle $(\mu_1,T_1)$ is called the infinitesimal of the 1-parameter formal deformation $(\frakg[[t]],\mu_t,T_t)$ of Rota-Baxter pre-Lie algebra $(\frakg,\mu,T)$.
	\end{defn}

In general, we can rewrite Equation~(\ref{Eq: deform eq for  products in RBA} ) and (\ref{Eq: Deform RB operator in RBA}) as
\begin{eqnarray} \label{Eq: general formal of deform product in RBA} \delta^2(\mu_n)(a\ot b\ot c) = \sum_{i=1}^{n-1} (\mu_i\circ \mu_{n-i})(a\ot b\ot c)- \sum_{i=1}^{n-1} (\mu_i\circ \mu_{n-i})(b\ot a\ot c)\end{eqnarray}
\begin{equation} \label{Eq: general formal of deform RBO in RBA}
\begin{array}{rcl}\partial^{1}(T_n)   +\Phi^2(\mu_n)&=& \sum_{i+j+k=n\atop 0 \leqslant i, j, k\leqslant n-1}	\mu_{i}\circ(T_j\ot T_{k})-\sum_{i+j+k=n\atop 0 \leqslant i, j, k\leqslant n-1} T_{i}\circ \mu_j\circ (\Id\ot T_{k})\\ &&-\sum_{i+j+k=n\atop 0 \leqslant i, j, k\leqslant n-1} T_{i}\circ\mu_j\circ (T_{k}\ot \Id)-\sum_{i+j=n\atop 0 \leqslant i, j\leqslant n-1}T_i\circ\mu_{j}.
\end{array}\end{equation}

\smallskip
\begin{defn}
Let $(\frakg[[t]],\mu_t,T_t)$ and $(\frakg[[t]],\mu_t',T_t')$ be two 1-parameter formal deformations of Rota-Baxter pre-Lie algebra $(\frakg,\mu,T)$. A formal isomorphism from $(\frakg[[t]],\mu_t',T_t')$ to $(\frakg[[t]], \mu_t, T_t)$ is a power series $\psi_t=\sum_{i=0}\psi_it^i: \frakg[[t]]\rightarrow \frakg[[t]]$, where $\psi_i: \frakg\rightarrow \frakg$ are linear maps with $\psi_0=\Id_\frakg$, such that:
\begin{eqnarray}\label{Eq: equivalent deformations}\psi_t\circ \mu_t' &=& \mu_t\circ (\psi_t\ot \psi_t),\\
\psi_t\circ T_t'&=&T_t\circ\psi_t. \label{Eq: equivalent deformations2}
	\end{eqnarray}
	In this case, we say that the two 1-parameter formal deformations $(\frakg[[t]], \mu_t,T_t)$ and
	$(\frakg[[t]],\mu_t',T_t')$ are  equivalent.
\end{defn}

\smallskip

Given a Rota-Baxter pre-Lie algebra $(\frakg,\mu,T)$, the power series $\mu_t,~~T_t$
with $\mu_i=\delta_{i,0}\mu,~~ T_i=\delta_{i,0}T$ make
$(\frakg[[t]],\mu_t,T_t)$ into a $1$-parameter formal deformation of
$(\frakg,\mu,T)$. Formal deformations equivalent to this one are called trivial.
\smallskip

\begin{thm}
The infinitesimals of two equivalent 1-parameter formal deformations of $(\frakg,\mu,T)$ are in the same cohomology class in $\rmH^\bullet_{\RBA}(\frakg)$.
\end{thm}

\begin{proof} Let $\psi_t:(\frakg[[t]],\mu_t',T_t')\rightarrow (\frakg[[t]],\mu_t,T_t)$ be a formal isomorphism.
	Expanding the identities and collecting coefficients of $t$, we get from Equations~(\ref{Eq: equivalent deformations}) and (\ref{Eq: equivalent deformations2}):
	\begin{eqnarray*}
		\mu_1'&=&\mu_1+\mu\circ(\Id\ot \psi_1)-\psi_1\circ\mu+\mu\circ(\psi_1\ot \Id),\\
		T_1'&=&T_1+T\circ\psi_1-\psi_1\circ T,
		\end{eqnarray*}
	that is, we have\[(\mu_1',T_1')-(\mu_1,T_1)=(\delta^1(\psi_1), -\Phi^1(\psi_1))=d^1(\psi_1,0)\in  \C^\bullet_{\RBA}(\frakg).\]
\end{proof}

\smallskip

\begin{defn}
	A Rota-Baxter pre-Lie algebra $(\frakg,\mu,T)$ is said to be rigid if every 1-parameter formal deformation is trivial.
\end{defn}

\begin{thm}
	Let $(\frakg,\mu,T)$ be a Rota-Baxter pre-Lie algebra of weight $\lambda$. If $\rmH^2_{\RBA}(\frakg)=0$, then $(\frakg,\mu,T)$ is rigid.
\end{thm}

\begin{proof}Let $(\frakg[[t]], \mu_t, T_t)$ be a $1$-parameter formal deformation of $(\frakg, \mu, T)$. By Proposition~\ref{Prop: Infinitesimal is 2-cocyle},
$(\mu_1, T_1)$ is a $2$-cocycle. By $\rmH^2_{\RBA}(\frakg)=0$, there exists a $1$-cochain $$(\psi_1', x) \in \C^1_\RBA(\frakg)= C^1_{\PLA}(\frakg)\oplus \Hom(k, \frakg)$$ such that
$(\mu_1, T_1) =  d^1(\psi_1', x), $
that is, $\mu_1=\delta^1(\psi_1')$ and $T_1=-\partial^0(x)-\Phi^1(\psi_1')$. Let $\psi_1=\psi_1'+\delta^0(x)$. Then
 $\mu_1= \delta^1(\psi_1)$ and $T_1=-\Phi^1(\psi_1)$, as it can be readily seen that $\Phi^1(\delta^0(x))=\partial^0(x)$.

Setting $\psi_t = \Id_\frakg -\psi_1t$, we have a deformation $(\frakg[[t]], \overline{\mu}_t, \overline{T}_t)$, where
$$\overline{\mu}_t=\psi_t^{-1}\circ \mu_t\circ (\psi_t\times \psi_t)$$
and $$\overline{T}_t=\psi_t^{-1}\circ T_t\circ \psi_t.$$
  It can be easily verify  that $\overline{\mu}_1=0, \overline{T}_1=0$. Then
    $$\begin{array}{rcl} \overline{\mu}_t&=& \mu+\overline{\mu}_2t^2+\cdots,\\
 T_t&=& T+\overline{T}_2t^2+\cdots.\end{array}$$
   By Equations~(\ref{Eq: general formal of deform product in RBA}) and (\ref{Eq: general formal of deform RBO in RBA}), we see that $(\overline{\mu}_2,  \overline{T}_2)$ is still a $2$-cocyle, so by induction, we can show that
  $ (\frakg[[t]], \mu_t , T_t) $ is equivalent to the trivial extension $(\frakg[[t]], \mu, T).$
Thus, $(\frakg,\mu,T)$ is rigid.

\end{proof}

\subsection{Formal deformations of Rota-Baxter operator  with product fixed}\

Let $(\frakg,\mu=\cdot, T)$ be a Rota-Baxter pre-Lie algebra of weight $\lambda$. Let us consider the case where  we only deform the Rota-Baxter operator with the product fixed. So    $\frakg[[t]]=\{\sum_{i=0}^\infty a_it^i\ | \ a_i\in \frakg, \forall i\geqslant 0\}$ is endowed with the product induced from that of $\frakg$, say,
$$(\sum_{i=0}^\infty a_it^i)(\sum_{j=0}^\infty b_jt^j)=\sum_{n=0}^\infty (\sum_{i+j=n\atop i,j\geqslant 0} a_ib_j)t^n.$$
Then $\frakg[[t]]$ becomes a flat $\bfk[[t]]$-algebra, whose product is still denoted by $\mu$.

 In this case, a  1-parameter formal deformation $(\mu_t,T_t)$ of    Rota-Baxter pre-Lie algebra $(\frakg, \mu,T)$  satisfies
 $\mu_i=0, \forall i\geqslant 1$. So Equation~(\ref{Eq: deform eq for  products in RBA}) degenerates and Equation~(\ref{Eq: Deform RB operator in RBA}) becomes
 \begin{eqnarray*}
 \mu\circ (T_t \ot T_t )&=& T_t\circ \Big(\mu\circ (\Id \ot T_t )+\mu \circ (T_t \ot \Id)+\lambda \mu \Big).
 \end{eqnarray*}
Expanding  these equations and comparing the coefficient of $t^n$, we obtain  that  $\{T_i\}_{i\geqslant0}$ have to  satisfy: for arbitrary $n\geqslant 0$,
\begin{eqnarray}\label{Eq: deform eq for RBO}	
\sum_{i+j=n\atop i, j\geqslant 0} 	\mu \circ(T_i\ot T_{j})&=&\sum_{i+j=n\atop i, j\geqslant 0} T_{ i}\circ \mu\circ (\Id\ot T_{j}) +\sum_{i+j=n\atop i, j\geqslant 0}T_{i}\circ\mu\circ (T_{j}\ot \Id)+\lambda T_n\circ\mu.
\end{eqnarray}

Obviously, when $n=0$,  Equation~(\ref{Eq: deform eq for RBO}) becomes
  exactly Equation~(\ref{Eq: Rota-Baxter relation}) defining Rota-Baxter operator $T=T_0$.

When $n=1$,    Equation~(\ref{Eq: deform eq for RBO}) has the form
 $$  	\mu \circ(T\ot T_1+T_1\otimes T)= T \circ \mu\circ (\Id\ot T_1) +T_1 \circ \mu\circ (\Id\ot T)+ T \circ\mu\circ (T_{1}\ot \Id)+  T_1 \circ\mu\circ (T\ot \Id)+  \lambda T_1\circ\mu$$
 which     says exactly that $\partial^1(T_1)=0\in \C^\bullet_{\RBO}(\frakg)$.
This proves the following result:
\begin{prop}
	Let $ T_t $ be a  1-parameter formal deformation of
	Rota-Baxter operator  $ T $ of weight $\lambda$. Then
	$ T_1 $ is a 1-cocycle in the cochain complex
	$\C_{\RBO}^\bullet(\frakg)$.
\end{prop}
This means that the cochain complex $\C^\bullet_\RBO(\frakg)$ controls formal deformations of Rota-Baxter operators.


\bigskip

\section{Abelian extensions}

 In this section, we study abelian extensions of Rota-Baxter pre-Lie algebras and show that they are classified by the second cohomology, as one would expect of a good cohomology theory.

 Notice that a vector space $M$ together with a linear transformation $T_M:M\to M$ is naturally a Rota-Baxter pre-Lie algebra where the multiplication on $M$ is defined to be $u\cdot v=0$ for all $u,v\in M.$

 \begin{defn}
 	An   abelian extension  of Rota-Baxter pre-Lie algebras is a short exact sequence of  morphisms of Rota-Baxter pre-Lie algebras
 \begin{eqnarray}\label{Eq: abelian extension} 0\to (M,\cdot, T_M)\stackrel{i}{\to} (\hat{\frakg},\hat{\cdot},  \hat{T})\stackrel{p}{\to} (\frakg, \cdot, T)\to 0,
 \end{eqnarray}
 that is, there exists a commutative diagram:
 	\[\begin{CD}
 		0@>>> {M} @>i >> \hat{\frakg} @>p >> \frakg @>>>0\\
 		@. @V {T_M} VV @V {\hat{T}} VV @V T VV @.\\
 		0@>>> {M} @>i >> \hat{\frakg} @>p >> \frakg @>>>0,
 	\end{CD}\]
 where the Rota-Baxter pre-Lie algebra $(M, \cdot, T_M)$	satisfies  $u\cdot v=0$ for all $u,v\in M.$

 We will call $(\hat{\frakg},\hat{\cdot}, \hat{T})$ an abelian extension of $(\frakg,\cdot, T)$ by $(M,\cdot, T_M)$.
 \end{defn}

 \begin{defn}
 	Let $(\hat{\frakg}_1,\hat{\cdot}_1, \hat{T}_1)$ and $(\hat{\frakg}_2,\hat{\cdot}_2, \hat{T}_2)$ be two abelian extensions of $(\frakg,\cdot, T)$ by $(M,\cdot, T_M)$. They are said to be  isomorphic  if there exists an isomorphism of Rota-Baxter pre-Lie algebras $\zeta:(\hat{\frakg}_1,\hat{T}_1)\rar (\hat{\frakg}_2,\hat{T}_2)$ such that the following commutative diagram holds:
 	\begin{eqnarray}\label{Eq: isom of abelian extension}\begin{CD}
 		0@>>> {(M,\cdot, T_M)} @>i >> (\hat{\frakg}_1,\hat{\cdot}_1,{\hat{T}_1}) @>p >> (\frakg,\cdot, T) @>>>0\\
 		@. @| @V \zeta VV @| @.\\
 		0@>>> {(M,\cdot, T_M)} @>i >> (\hat{\frakg}_2,\hat{\cdot}_2,{\hat{T}_2}) @>p >> (\frakg,\cdot, T) @>>>0.
 	\end{CD}\end{eqnarray}
 \end{defn}

 A   section of an abelian extension $(\hat{\frakg},\hat{\cdot}, {\hat{T}})$ of $(\frakg,\cdot, T)$ by $(M,\cdot, T_M)$ is a linear map $s:\frakg\rar \hat{\frakg}$ such that $p\circ s=\Id_\frakg$.

 We will show that isomorphism classes of  abelian extensions of $(\frakg,\cdot, T)$ by $(M,\cdot, T_M)$ are in bijection with the second cohomology group   ${\rmH}_{\RBA}^2(\frakg,M)$.

 \bigskip

Let    $(\hat{\frakg},\hat{\cdot}, \hat{T})$ be  an abelian extension of $(\frakg,\cdot, T)$ by $(M,\cdot, T_M)$ having the form Equation~\eqref{Eq: abelian extension}. Choose a section $s:\frakg\rar \hat{\frakg}$. We   define
 $$
 a\cdot m:=s(a)\cdot m,\quad m\cdot a:=m\cdot s(a), \quad \forall a\in \frakg, m\in M.
 $$
 \begin{prop}\label{Prop: new RB bimodules from abelian extensions}
 	With the above notations, $(M, \cdot, T_M)$ is a Rota-Baxter  bimodule over $(\frakg,\cdot, T)$.
 \end{prop}
 \begin{proof}
 	For arbitrary $a,~b\in \frakg,\,m\in M$, since $s(a\cdot b)-s(a)\hat{\cdot} s(b)\in M$ implies $s(a\cdot b)\cdot m=(s(a)\hat{\cdot} s(b))\cdot m$, we have
 	\begin{eqnarray*}
&& a\cdot (b\cdot m) -(a\cdot b)\cdot  m\\
&=&s(a)\cdot (s(b)\cdot m)-s(a\cdot b)\cdot m\\
&=&  s(b)\cdot(s(a)\cdot m)-s(b\cdot a)\cdot m\\
&=& b\cdot (a\cdot m) -(b\cdot a)\cdot  m.
  \end{eqnarray*}
 	Hence,  this gives a $\frakg$-bimodule structure and other equation is similar.

 Moreover, ${\hat{T}}(s(a))-s(T(a))\in M$ means that  ${\hat{T}}(s(a))\cdot m=s(T(a))\cdot m$. Thus we have
 	\begin{align*}
 		T(a)\cdot T_M(m)&=s(T(a))\cdot T_M(m)\\
 		&=\hat{T}(s(a))\cdot T_M(m)\\
 		&=\hat{T}(\hat{T}(s(a))\cdot m+s(a)\cdot T_M(m)+\lambda s(a)\cdot m)\\
 		&=T_M(T(a)\cdot m+a\cdot T_M(m)+\lambda a\cdot m)
 	\end{align*}
 	It is similar to see $T_M(m)\cdot T(a)=T_M(T_M(m)\cdot a+m\cdot T(a)+\lambda m\cdot a)$.

 	Hence, $(M, \cdot, T_M)$ is a  Rota-Baxter  bimodule over $(\frakg, \cdot,  T)$.
 \end{proof}

 We  further  define linear maps $\psi:\frakg\ot \frakg\rar M$ and $\chi:\frakg\rar M$ respectively by
 \begin{align*}
 	\psi(a\ot b)&=s(a)\hat{\cdot}s(b)-s(a\cdot b),\quad\forall a,b\in \frakg,\\
 	\chi(a)&={\hat{T}}(s(a))-s(T(a)),\quad\forall a\in \frakg.
 \end{align*}

 \begin{prop}\label{prop:2-cocycle}
 	 The pair
 	$(\psi,\chi)$ is a 2-cocycle  of   Rota-Baxter pre-Lie algebra $(\frakg,\cdot, T)$ with  coefficients  in the Rota-Baxter bimodule $(M,\cdot, T_M)$ introduced in Proposition~\ref{Prop: new RB bimodules from abelian extensions}.
 \end{prop}

 	The proof is by direct computations, so it is left to the reader.

 The choice of the section $s$ in fact determines a splitting
 $$\xymatrix{0\ar[r]&  M\ar@<1ex>[r]^{i} &\hat{\frakg}\ar@<1ex>[r]^{p} \ar@<1ex>[l]^{t}& \frakg \ar@<1ex>[l]^{s} \ar[r] & 0}$$
 subject to $t\circ i=\Id_M, t\circ s=0$ and $ it+sp=\Id_{\hat{\frakg}}$.
 Then there is an induced isomorphism of vector spaces
 $$\left(\begin{array}{cc} p& t\end{array}\right): \hat{\frakg}\cong   \frakg\oplus M: \left(\begin{array}{c} s\\ i\end{array}\right).$$
We can  transfer the Rota-Baxter pre-Lie algebra structure on $\hat{\frakg}$ to $\frakg\oplus M$ via this isomorphism.
  It is direct to verify that this  endows $\frakg\oplus M$ with a multiplication $\cdot_\psi$ and an Rota-Baxter operator $T_\chi$ defined by
 \begin{align}
 	\label{eq:mul}(a,m)\cdot_\psi(b,n)&=(a\cdot b,a\cdot n+m\cdot b+\psi(a\otimes b)),\,\forall a,b\in \frakg,\,m,n\in M,\\
 	\label{eq:dif}T_\chi(a,m)&=(T(a),\chi(a)+T_M(m)),\,\forall a\in \frakg,\,m\in M.
 \end{align}
 Moreover, we get an abelian extension
 $$0\to (M, \cdot, T_M)\stackrel{\left(\begin{array}{cc} s& i\end{array}\right) }{\to} (\frakg\oplus M, T_\chi)\stackrel{\left(\begin{array}{c} p\\ t\end{array}\right)}{\to} (\frakg, \cdot, T)\to 0$$
 which is easily seen to be  isomorphic to the original one \eqref{Eq: abelian extension}.

 \medskip

 Now we investigate the influence of different choices of   sections.

 \begin{prop}\label{prop: different sections give}
 \begin{itemize}
 \item[(i)] Different choices of the section $s$ give the same  Rota-Baxter bimodule structures on $(M, T_M)$;

 \item[(ii)]   the cohomological class of $(\psi,\chi)$ does not depend on the choice of sections.

 \end{itemize}

 \end{prop}
 \begin{proof}Let $s_1$ and $s_2$ be two distinct sections of $p$.
  We define $\gamma:\frakg\rar M$ by $\gamma(a)=s_1(a)-s_2(a)$.

  Since the Rota-Baxter pre-Lie algebra $(M, \cdot,  T_M)$	satisfies  $u\cdot v=0$ for all $u,v\in M$,
  $$s_1(a)\cdot m= s_2(a)\cdot m+\gamma(a)\cdot m=s_2(a)\cdot m.$$ So different choices of the section $s$ give the same  Rota-Baxter bimodule structures on $(M, \cdot, T_M)$;

  We   show that the cohomological class of $(\psi,\chi)$ does not depend on the choice of sections.   Then
 	\begin{align*}
 		\psi_1(a,b)&=s_1(a)\cdot  s_1(b)-s_1(a\cdot b)\\
 		&=(s_2(a)+\gamma(a))\cdot (s_2(b)+\gamma(b))-(s_2(a\cdot b)+\gamma(a\cdot b))\\
 		&=(s_2(a)\cdot s_2(b)-s_2(a\cdot b))+s_2(a)\cdot\gamma(b)+\gamma(a)\cdot s_2(b)-\gamma(a\cdot b)\\
 		&=(s_2(a)\cdot s_2(b)-s_2(a\cdot b))+a\cdot \gamma(b)+\gamma(a)\cdot b-\gamma(a\cdot b)\\
 		&=\psi_2(a,b)+\delta(\gamma)(a,b)
 	\end{align*}
 	and
 	\begin{align*}
 		\chi_1(a)&={\hat{T}}(s_1(a))-s_1(T(a))\\
 		&={\hat{T}}(s_2(a)+\gamma(a))-(s_2(T(a))+\gamma(T(a)))\\
 		&=({\hat{T}}(s_2(a))-s_2(T(a)))+{\hat{T}}(\gamma(a))-\gamma(T(a))\\
 		&=\chi_2(a)+T_M(\gamma(a))-\gamma(T(a))\\
 		&=\chi_2(a)-\Phi^1(\gamma)(a).
 	\end{align*}
 	That is, $(\psi_1,\chi_1)=(\psi_2,\chi_2)+d^1(\gamma)$. Thus $(\psi_1,\chi_1)$ and $(\psi_2,\chi_2)$ form the same cohomological class  {in $\rmH_{\RBA}^2(\frakg,M)$}.

 \end{proof}

 We show now the isomorphic abelian extensions give rise to the same cohomology classes.
 \begin{prop}Let $M$ be a vector space and  $T_M\in\End_\bfk(M)$. Then $(M,\cdot,  T_M)$ is a Rota-Baxter pre-Lie algebra with trivial multiplication.
 Let $(\frakg,\cdot, T)$ be a  Rota-Baxter pre-Lie algebra.
 Two isomorphic abelian extensions of Rota-Baxter pre-Lie algebra $(\frakg,\cdot,  T)$ by  $(M, \cdot,  T_M)$  give rise to the same cohomology class  in $\rmH_{\RBA}^2(\frakg,M)$.
 \end{prop}
 \begin{proof}
  Assume that $(\hat{\frakg}_1, \hat{\cdot}_1,  {\hat{T}_1})$ and $(\hat{\frakg}_2, \hat{\cdot}_2, {\hat{T}_2})$ are two isomorphic abelian extensions of $(\frakg,\cdot, T)$ by $(M,\cdot, T_M)$ as is given in \eqref{Eq: isom of abelian extension}. Let $s_1$ be a section of $(\hat{\frakg}_1,\hat{\cdot}_1,  {\hat{T}_1})$. As $p_2\circ\zeta=p_1$, we have
 	\[p_2\circ(\zeta\circ s_1)=p_1\circ s_1=\Id_{\frakg}.\]
 	Therefore, $\zeta\circ s_1$ is a section of $(\hat{\frakg}_2,\hat{\cdot}_2, {\hat{T}_2})$. Denote $s_2:=\zeta\circ s_1$. Since $\zeta$ is a homomorphism of Rota-Baxter  algebras such that $\zeta|_M=\Id_M$, $\zeta(a\cdot m)=\zeta(s_1(a)\cdot m)=s_2(a)\cdot m=a\cdot m$, so $\zeta|_M: M\to M$ is compatible with the induced  Rota-Baxter bimodule structures.
 We have
 	\begin{align*}
 		\psi_2(a\ot b)&=s_2(a)\hat{\cdot}_2 s_2(b)-s_2(a\cdot b)=\zeta(s_1(a))\hat{\cdot}_2\zeta(s_1(b))-\zeta(s_1(a\cdot b))\\
 		&=\zeta(s_1(a)\hat{\cdot}_1s_1(b)-s_1(a\cdot b))=\zeta(\psi_1(a,b))\\
 		&=\psi_1(a,b)
 	\end{align*}
 	and
 	\begin{align*}
 		\chi_2(a)&={\hat{T}_2}(s_2(a))-s_2(T(a))={\hat{T}_2}(\zeta(s_1(a)))-\zeta(s_1(T(a)))\\
 		&=\zeta({\hat{T}_1}(s_1(a))-s_1(T(a)))=\zeta(\chi_1(a))\\
 		&=\chi_1(a).
 	\end{align*}
 	Consequently, two isomorphic abelian extensions give rise to the same element in {$\rmH_{\RBA}^2(\frakg,M)$}.
\end{proof}
 \bigskip

 Now we consider the reverse direction.

 Let $(M, \cdot, T_M)$ be a Rota-Baxter bimodule over Rota-Baxter pre-Lie algebra $(\frakg, \cdot,  T)$, given two linear maps  $\psi:\frakg\ot \frakg\rar M$ and $\chi:\frakg\rar M$, one can define  a multiplication $\cdot_\psi$ and an operator $T_\chi$  on  $\frakg\oplus M$ by Equations~(\ref{eq:mul})(\ref{eq:dif}).
 The following fact is important:
 \begin{prop}\label{prop:2-cocycle}
 	The triple $(\frakg\oplus M,\cdot_\psi,T_\chi)$ is a Rota-Baxter pre-Lie algebra   if and only if
 	$(\psi,\chi)$ is a 2-cocycle  of the Rota-Baxter pre-Lie algebra $(\frakg,\cdot, T)$ with  coefficients  in $(M,\cdot, T_M)$.
 In this case,    we obtain an abelian extension
  $$0\to (M, \cdot, T_M)\stackrel{\left(\begin{array}{cc} 0& \Id  \end{array}\right) }{\to} (\frakg\oplus M, T_\chi)\stackrel{\left(\begin{array}{c} \Id\\ 0\end{array}\right)}{\to} (\frakg,\cdot,  T)\to 0,$$
  and the canonical section $s=\left(\begin{array}{cc}   \Id  & 0\end{array}\right): (\frakg, T)\to (\frakg\oplus M, T_\chi)$ endows $M$ with the original Rota-Baxter bimodule structure.
 \end{prop}
 \begin{proof}
 	If $(\frakg\oplus M,\cdot_\psi,T_\chi)$ is a Rota-Baxter pre-Lie algebra, then $\cdot_\psi$ implies
 	\begin{align}\label{eq:mc}
 		&a\cdot \psi(b\ot c)-\psi(a\cdot b\ot c)+\psi(a\ot b\cdot c)-\psi(a\ot b)\cdot c\\
 &=b\cdot \psi(a\ot c)-\psi(b\cdot a\ot c)+\psi(b\ot a\cdot c)-\psi(b\ot a)\cdot c\nonumber,
 	\end{align}
 	which means $\delta^2(\phi)=0$ in $\C^\bullet(\frakg,M)$.
 	Since $T_\chi$ is a Rota-Baxter operator,
 	for arbitrary $a,b\in \frakg, m,n\in M$, we have
 	$$T_\chi((a,m))\cdot_\psi T_\chi((b,n))=T_\chi\Big(T_\chi(a,m)\cdot_\psi(b,n)+(a,m)\cdot_\psi T_\chi(b,n)+\lambda(a,m)\cdot_\psi(b,n)\Big)$$
 	Then $\chi,\psi$ satisfy the following equations:
 	\begin{align*}
 		&T(a)\cdot \chi(b)+\chi(a)\cdot T(b)+\psi(T(a)\ot T(b))\\
 		=&T_M(\chi(a)\cdot b)+T_M(\psi(T(a)\ot b))+\chi(T(a)\cdot b)\\
 		&+T_M(a\chi(b))+T_M(\psi(a\ot T(b)))+\chi(a\cdot T(b))\\
 		&+\lambda T_M(\psi(a\ot b))+\lambda \chi(a\cdot b).
 	\end{align*}
 	
 	That is,
 	\[ \partial^1(\chi)+\Phi^2(\psi)=0.\]
 	Hence, $(\psi,\chi)$ is a  2-cocycle.
 	
 	 Conversely, if $(\psi,\chi)$ is a 2-cocycle, one can easily check that $(\frakg\oplus M,\cdot_\psi,T_\chi)$ is a  Rota-Baxter pre-Lie algebra.

  The last statement is clear.
 \end{proof}

 Finally, we show the following result:
 \begin{prop}
 	Two cohomologous $2$-cocyles give rise to isomorphic abelian extensions.
 \end{prop}
 \begin{proof}

 	  Given two 2-cocycles $(\psi_1,\chi_1)$ and $(\psi_2,\chi_2)$, we can construct two abelian extensions $(\frakg\oplus M,\cdot_{\psi_1},T_{\chi_1})$ and  $(\frakg\oplus M,\cdot_{\psi_2},T_{\chi_2})$ via Equations~\eqref{eq:mul} and \eqref{eq:dif}. If they represent the same cohomology  class {in $\rmH_{\RBA}^2(\frakg,M)$}, then there exists two linear maps $\gamma_0:k\rightarrow M, \gamma_1:\frakg\to M$ such that $$(\psi_1,\chi_1)=(\psi_2,\chi_2)+(\delta^1(\gamma_1),-\Phi^1(\gamma_1)-\partial^0(\gamma_0)).$$
 	Notice that $\partial^0=\Phi^1\circ\delta^0$. Define $\gamma: \frakg\rightarrow M$ to be $\gamma_1+\delta^0(\gamma_0)$. Then $\gamma$ satisfies
 	\[(\psi_1,\chi_1)=(\psi_2,\chi_2)+(\delta^1(\gamma),-\Phi^1(\gamma)).\]
 	
 	Define $\zeta:\frakg\oplus M\rar \frakg\oplus M$ by
 	\[\zeta(a,m):=(a, -\gamma(a)+m).\]
 	Then $\zeta$ is an isomorphism of these two abelian extensions $(\frakg\oplus M,\cdot_{\psi_1},T_{\chi_1})$ and  $(\frakg\oplus M,\cdot_{\psi_2},T_{\chi_2})$.
 \end{proof}

 \bigskip

\section{Skeletal Rota-Baxter  pre-Lie 2-algebras and crossed modules}

In this section, we introduce the notion of a Rota-Baxter pre-Lie 2-algebra and
show that skeletal Rota-Baxter pre-Lie 2-algebras are classified by 3-cocycles of Rota-Baxter pre-Lie
algebras.

\subsection{Skeletal Rota-Baxter  pre-Lie 2-algebras}

In the following, we first recall the definition of pre-Lie 2-algebras from \cite{S19}.
\begin{defn}\label{Def: preLie 2-algebra}
A pre-Lie 2-algebra $\mathcal{G}=\mathfrak{g}_0\oplus \mathfrak{g}_1$ over $\mathbf{k}$ consists of the following data:

$\bullet$~  a  linear map: $d: \mathfrak{g}_1\rightarrow \mathfrak{g}_0$,

$ \bullet$~   bilinear maps $l_2 : \mathfrak{g}_i \otimes \mathfrak{g}_j \rightarrow \mathfrak{g}_{i+j}$, where $0 \leq i + j \leq 1$,

$ \bullet$ ~  a trilinear map $l_3 : \wedge^2\mathfrak{g}_0 \otimes \mathfrak{g}_0 \rightarrow \mathfrak{g}_1$,\\
such that for all $x, y, z, w \in \mathfrak{g}_0$ and $\alpha, \beta \in \mathfrak{g}_1$, the following equalities are satisfied:
\begin{eqnarray*}
&&(a)~dl_2(x, \alpha)=l_2(x, d(\alpha)),\\
&&(b)~dl_2(\alpha, x)=l_2(d(\alpha), x),\\
&&(c)~l_2(d(\alpha), \beta)=l_2(\alpha, d(\beta)),\\
&&(e_1)~dl_3(x, y, z)=l_2(x, l_2(y, z))-l_2(l_2(x, y), z)-l_2(y, l_2(x, z))+l_2(l_2(y, x), z),\\
&&(e_2)~l_3(x, y, d(\alpha))=l_2(x, l_2(y, \alpha))-l_2(l_2(x, y), \alpha)-l_2(y, l_2(x, \alpha))+l_2(l_2(y, x), \alpha),\\
&& (e_3)~l_3( d(\alpha), x, y)=l_2(\alpha, l_2(x, y))-l_2(l_2(\alpha, x), y)-l_2(x, l_2(\alpha, y))+l_2(l_2(x, \alpha), y),\\
&&(f)~ l_2(w,l_3(x, y, z))- l_2(x,l_3(w, y, z)) + l_2(y,l_3(w, x, z)) + l_2(l_3(x, y, w), z)\\
&&~~~~-l_2(l_3(w, y, x), z) + l_2(l_3(w, x, y), z) - l_3(x, y,l_2(w, z)) + l_3(w, y,l_2(x, z))\\
&&~~~~-l_3(w, x,l_2(y, z))-l_3(l_2(w, x)- l_2(x, w), y, z) + l_3(l_2(w, y)- l_2(y, w), x, z)\\
&&~~~~-l_3(l_2(x, y)-l_2(y, x), w, z) = 0.
\end{eqnarray*}
\end{defn}
inspired by \cite{JS21}, we give the notion of a Rota-Baxter pre-Lie 2-algebra of weight $\lambda$.
\begin{defn}\label{Def: RBpreLie 2-algebra}
A Rota-Baxter pre-Lie 2-algebra of weight $\lambda$ consists of a pre-Lie 2-algebra $\mathcal{G} = (\mathfrak{g}_0, \mathfrak{g}_1, d, l_2, l_3)$
and a  Rota-Baxter operator $\Theta = (T_0, T_1, T_2)$ on $\mathcal{G}$, where $T_0:\mathfrak{g}_0\rightarrow \mathfrak{g}_0, ~~T_1: \mathfrak{g}_1\rightarrow \mathfrak{g}_1$ and $T_2: \mathfrak{g}_0\otimes \mathfrak{g}_0\rightarrow \mathfrak{g}_1$ satisfying  the following equalities:
\begin{eqnarray*}
	(i)&& T_0\circ d=d\circ T_1\\
(ii)&& T_0(l_2(T_0(x), y)+l_2(x, T_0(y))+\lambda l_2(x, y))-l_2(T_0(x), T_0(y))=dT_2(x, y);\\
(iii)&&T_1(l_2(T_1(\alpha), x)+l_2(\alpha, T_0(x))+\lambda l_2(\alpha, x))-l_2(T_1(\alpha), T_0(x))=T_2(d(\alpha), x);\\
(iv)&&  T_1(l_2(x, T_1(\alpha))+l_2( T_0(x), \alpha)+\lambda l_2( x, \alpha))-l_2( T_0(x), T_1(\alpha))=T_2(x, d(\alpha));\\
(v)&& l_2(T_0(x_1), T_2(x_2, x_3))-l_2(T_0(x_2), T_2(x_1, x_3))+l_2( T_2(x_2, x_1), T_0(x_3))-l_2( T_2(x_1, x_2), T_0(x_3))\\
&&-T_2(x_2, l_2(T_0(x_1), x_3)+l_2(x_1, T_0(x_3))+\lambda l_2(x_1, x_3))+T_2(x_1, l_2(T_0(x_2), x_3)+l_2( x_2, T_0(x_3))\\
&&+\lambda l_2(x_2, x_3))-T_2(l_2(T_0(x_1), x_2)+l_2(x_1, T_0(x_2))+\lambda l_2(x_1, x_2), x_3)+T_2(l_2(T_0(x_2), x_1)\\
&&+l_2(x_2, T_0(x_1))+\lambda l_2(x_2, x_1), x_3)+l_3(T_0(x_1),T_0(x_2), T_0(x_3) )-T_1(l_3(x_1, x_2, x_3))=0.
\end{eqnarray*}
\end{defn}
We will denote a Rota-Baxter pre-Lie 2-algebra by $(\mathcal{G}, \Theta)$. A Rota-Baxter pre-Lie 2-algebra is said to
be {\bf skeletal} if $d = 0$.    A Rota-Baxter pre-Lie 2-algebra is said to
be {\bf strict} if $l_3 = 0, T_2=0$.

Let $(\mathcal{G}, \Theta)$ be a skeletal Rota-Baxter pre-Lie 2-algebra of weight $\lambda$. Then $\mathcal{G}$ is a  skeletal pre-Lie 2-algebra. Therefore, we have
\begin{eqnarray}
&&~0=l_2(x, l_2(y, z))-l_2(l_2(x, y), z)-l_2(y, l_2(x, z))+l_2(l_2(y, x), z),\\
&&~0=l_2(x, l_2(y, \alpha))-l_2(l_2(x, y), \alpha)-l_2(y, l_2(x, \alpha))+l_2(l_2(y, x), \alpha),\\
&& 0=l_2(\alpha, l_2(x, y))-l_2(l_2(\alpha, x), y)-l_2(x, l_2(\alpha, y))+l_2(l_2(x, \alpha), y),\\
&&0= l_2(w,l_3(x, y, z))- l_2(x,l_3(w, y, z)) + l_2(y,l_3(w, x, z)) + l_2(l_3(x, y, w), z)\label{pre-lie 2-algebra relation}\\
&&~~~~-l_2(l_3(w, y, x), z) + l_2(l_3(w, x, y), z) - l_3(x, y,l_2(w, z)) + l_3(w, y,l_2(x, z))\nonumber\\
&&~~~~-l_3(w, x,l_2(y, z))-l_3(l_2(w, x)- l_2(x, w), y, z) + l_3(l_2(w, y)- l_2(y, w), x, z)\nonumber\\
&&~~~~-l_3(l_2(x, y)-l_2(y, x), w, z), \nonumber
\end{eqnarray}
for all $x,~y,~z,~w\in \mathfrak{g}_0$ and $\alpha\in \mathfrak{g}_1$.

\begin{thm}
There is a one-to-one correspondence between skeletal  Rota-Baxter pre-Lie 2-algebras of weight $\lambda$ and 3-cocycles of  Rota-Baxter pre-Lie 2-algebras of weight $\lambda$.

\end{thm}
\begin{proof}
Let $(\mathcal{G}, \Theta)$ be a skeletal Rota-Baxter pre-Lie 2-algebra of weight $\lambda$. Then $(\mathfrak{g}_0, l_2, T_0)$ is a Rota-Baxter pre-Lie algebras  of weight $\lambda$. Define linear maps
$f: \mathfrak{g}_0\otimes \mathfrak{g}_0 \otimes \mathfrak{g}_0\rightarrow \mathfrak{g}_1$ and $\theta: \mathfrak{g}_0 \otimes \mathfrak{g}_0\rightarrow \mathfrak{g}_1$ by
\begin{eqnarray*}
& f(x, y, z)=l_3(x, y, z),\\
& \theta(x, y)=T_2(x, y).
\end{eqnarray*}
By (iv) in Definition \ref{Def: RBpreLie 2-algebra}, we have
\begin{eqnarray}\label{RB pre-lie 2-algebra relation}
&& -T_0(x_1)\cdot\theta(x_2, x_3)+T_0(x_2)\cdot\theta(x_1, x_3))- \theta(x_2, x_1)\cdot T_0(x_3)+\theta(x_1, x_2)\cdot T_0(x_3)\nonumber\\
&&+\theta(x_2, T_0(x_1)\cdot x_3+x_1\cdot T_0(x_3)+\lambda x_1 \cdot x_3)-\theta(x_1, T_0(x_2)\cdot x_3+x_2\cdot T_0(x_3)\\
&&+\lambda x_2\cdot x_3)+\theta(T_0(x_1)\cdot x_2+x_1\cdot T_0(x_2)+\lambda x_1\cdot x_2, x_3)-\theta(T_0(x_2)\cdot x_1\nonumber\\
&&+x_2\cdot T_0(x_1)+\lambda x_2\cdot x_1, x_3)+f(T_0(x_1),T_0(x_2), T_0(x_3) )-T_1(f(x_1, x_2, x_3))=0.\nonumber
\end{eqnarray}
 By (\ref{pre-lie 2-algebra relation}) and (\ref{RB pre-lie 2-algebra relation}),  we have $d^3(f, \theta)=0$. Thus,  $(f, \theta)$ is a 3-cocycle of  Rota-Baxter pre-Lie 2-algebras of weight $\lambda$.

The proof of other direction is similar, and  the proof is  finished.
\end{proof}

\subsection{Crossed modules}\

Now we turn to the study on strict Rota-Baxter pre-Lie 2-algebras of weight $\lambda$. First we introduce the notion of crossed
modules of  Rota-Baxter pre-Lie 2-algebras of weight $\lambda$.

\begin{defn}
A crossed module of Rota-Baxter pre-Lie 2-algebras of weight $\lambda$ is a quintuple $((\mathfrak{g}_0, \cdot_0), (\mathfrak{g}_1, \cdot_1), d, (S, P), (T_0, T_1, T_2=0))$,
where $(\mathfrak{g}_0, \cdot_0, T_0)$ is a  Rota-Baxter pre-Lie algebra of weight $\lambda$ and $(\mathfrak{g}_1, \cdot_1)$ is a  pre-Lie algebra, $d: \mathfrak{g}_1\rightarrow \mathfrak{g}_0$ is a homomorphism of pre-Lie algebras,   and $(\mathfrak{g}_1, S, P, T_1)$ is a Rota-Baxter bimodule  over $(\mathfrak{g}_0, \cdot_0, T_0)$ such that for all $x \in \mathfrak{g}_0$ and $\alpha,~~\beta \in \mathfrak{g}_1$, the following
equalities are satisfied:
\begin{eqnarray*}
(C1) &&d(S(x)\alpha) = x \cdot_0 d\alpha,~~ d(P(x)\alpha) = (d\alpha) \cdot_0 x,~~d\circ T_1=T_0\circ d; \\
(C2) && S(d\alpha)\beta = P(d\beta)\alpha = \alpha \cdot_1\beta.
\end{eqnarray*}

\end{defn}
Let $(\mathcal{G}, \Theta)$ be a strict Rota-Baxter pre-Lie 2-algebra of weight $\lambda$. Then $\mathcal{G}$ is a  strict pre-Lie 2-algebra. Therefore, we have
\begin{eqnarray}
&&~dl_2(x, \alpha)=l_2(x, d(\alpha)),\\
&&~dl_2(\alpha, x)=l_2(d(\alpha), x),\label{pre-lie algebra 1}\\
&&l_2(d(\alpha), \beta)=l_2(\alpha, d(\beta)),\\
&&~0=l_2(x, l_2(y, z))-l_2(l_2(x, y), z)-l_2(y, l_2(x, z))+l_2(l_2(y, x), z),\\
&&~0=l_2(x, l_2(y, \alpha))-l_2(l_2(x, y), \alpha)-l_2(y, l_2(x, \alpha))+l_2(l_2(y, x), \alpha),\label{pre-lie algebra 2}\\
&& ~0=l_2(\alpha, l_2(x, y))-l_2(l_2(\alpha, x), y)-l_2(x, l_2(\alpha, y))+l_2(l_2(x, \alpha), y),\label{RB bimodule 1}
\end{eqnarray}
for all $x,~~ y\in \mathfrak{g}_0$ and $\alpha\in \mathfrak{g}_1$.
\begin{thm}
There is a one-to-one correspondence between strict Rota-Baxter pre-Lie 2-algebras of weight $\lambda$ and crossed
modules of Rota-Baxter pre-Lie 2-algebras of weight $\lambda$.
\end{thm}
\begin{proof}
Let $(\mathcal{G}, \Theta)$ be a strict Rota-Baxter pre-Lie 2-algebra of weight $\lambda$. We construct a crossed module of
strict Rota-Baxter pre-Lie 2-algebra of weight $\lambda$ as follows. Obviously, $(\mathfrak{g}_0, \cdot_0, T_0)$ is a  Rota-Baxter pre-Lie algebra of weight $\lambda$. Define a multiplication $\cdot_1$ on $\mathfrak{g}_1$ by
\begin{eqnarray*}\label{RB bimodule 2}
\alpha\cdot_1 \beta = l_2(d\alpha,  \beta) =l_2( \alpha,  d\beta).
\end{eqnarray*}
By conditions (\ref{pre-lie algebra 1}) and (\ref{pre-lie algebra 2}). Then    $(\mathfrak{g}_1, \cdot_1)$ is a pre-Lie algebra.  Obviously, we deduce that $d$ is a
homomorphism between pre-Lie algebras. Define $S,~ P: \mathfrak{g}_0 \rightarrow gl(\mathfrak{g}_1)$ by
\begin{eqnarray}
S(x)\alpha = l_2(x, \alpha),~~ P(x)\alpha = l_2(x, \alpha).
\end{eqnarray}
By conditions (ii),  (iii) in Definition \ref{Def: RBpreLie 2-algebra} and (\ref{RB bimodule 1}), (\ref{RB bimodule 2}).  Then   $(\mathfrak{g}_1, S, P, T_1)$ is a Rota-Baxter bimodule  over $(\mathfrak{g}_0, \cdot_0, T_0)$.
Thus $((\mathfrak{g}_0, \cdot_0), (\mathfrak{g}_1, \cdot_1), d, (S, P), T_0, T_1, T_2=0)$ is a crossed module of Rota-Baxter pre-Lie 2-algebras of weight $\lambda$.

Conversely, a crossed module of Rota-Baxter pre-Lie 2-algebras of weight $\lambda$ $((\mathfrak{g}_0, \cdot_0), (\mathfrak{g}_1, \cdot_1), d,$  \\
$(S, P), T_0, T_1, T_2=0)$ gives rise to a strict Rota-Baxter pre-Lie 2-algebra of weight $\lambda$ $(\mathfrak{g}_0,\mathfrak{g}_1, d, l_2, l_3 = 0, T_0, T_1, T_2=0)$, where $l_2 : \mathfrak{g}_i \otimes \mathfrak{g}_j \rightarrow \mathfrak{g}_{i+j} ,~~ 0 \leq i + j \leq 1$ is given by
\begin{eqnarray*}
l_2(x, y) = x \cdot_0 y,~~ l_2(x, \alpha)= S(x)\alpha,~~ l_2(\alpha, x) = P(x)\alpha.
\end{eqnarray*}
Direct verification shows that $(\mathfrak{g}_0, \mathfrak{g}_1, d, l_2, l_3=0, T_0, T_1, T_2=0)$  is a strict Rota-Baxter pre-Lie 2-algebra of weight $\lambda$.
\end{proof}

\bigskip

 \section{From Rota-Baxter associative algebras to Rota-Baxter pre-Lie algebras}

We  recall the   cohomology of  Rota-Baxter associative algebras with any weight  which was first defined by  W. Kai and G. Zhou \cite{WZ21}. We first recall the definition of Rota-Baxter associative algebras and Rota-Baxter	bimodules.
\begin{defn}\cite{WZ21}
	\begin{itemize}
		\item [(i)]Given  an algebra  $A$, it is called  a Rota-Baxter associative algebra  of weight $\lambda$, if there is a linear
		operator $T: \frakg \rightarrow \frakg $ subject to
		\begin{eqnarray}\label{Eq: Rota-Baxter relation}
			T(a)\pl  T(b)=T\big(a\pl  T(b)+T(a)\pl  b+\lambda\  a\pl
			b\big)\end{eqnarray}	
		for arbitrary $a,~b \in \frakg $;
		\item [(ii)]Let $A$ be a Rota-Baxter associative algebra  of weight $\lambda$ and $(M, \cdot)$ be a bimodule over
		algebra $A$. We say that $(M, \cdot, T_M)$ is a bimodule over Rota-Baxter associative algebra $(A, \mu, T)$  of weight $\lambda$ or a Rota-Baxter	bimodule if $(M, \cdot)$ is endowed with a linear operator $T_M : M \rightarrow M$ such that the following equations
		\begin{eqnarray}
			&&T(a)\cdot T_M(m) = T_M(a\cdot T_M(m) + T(a)\cdot m + \lambda a\cdot m),\\
			&&T_M(m)\cdot T(a) = T_M(m\cdot T(a) + T_M(m)\cdot a + \lambda m\cdot a)
		\end{eqnarray}
		hold for arbitrary  $a \in \frakg$ and $m \in M$.
	\end{itemize}
\end{defn}
For a Rota-Baxter associative algebra  $(A,\ T)$ of weight $\lambda$    and a Rota-Baxter	bimodule $(M,\  \cdot,\  T_M)$, we can define a new  bimodule structure
$\left( (A,\ \star),( M,\ \rhd ,\ \lhd) \right) $ as follow: for arbitrary $a,b\in A ,m\in M$,
\begin{eqnarray}
	a\star  b:&=&a\pl  T(b)+T(a)\pl  b+\lambda a\pl  b,\\
	a\rhd m:&=& T(a)\pl m-T_M(a \pl m),\\
	m\lhd a:&=& m\pl T(a)-T_M(m \pl a).
\end{eqnarray}

 The cochain complex of associative version Rota-Baxter operator $T$  over a   Rota-Baxter  bimoldule $(M,\ T_M)$, denoted by $\C^\bullet_{\RBO_{\lambda}}(A , M)=\C_{\Alg}^\bullet(A _\star , {_\rhd
 	M_\lhd})=\bigoplus\limits_{n=0}^\infty \C^n_{\Alg}(A _\star , {_\rhd
	M_\lhd})$. For $n\geq 0$, $ \C^n_{\Alg}(A _\star , {_\rhd
	M_\lhd})=\Hom(A^{\otimes n}, M)$ and its differential $$\partial^n:
\C^n_{\mathrm{Alg}}(A _\star ,\  _\rhd M_\lhd)\rightarrow  \C^{n+1}_{\mathrm{Alg}}(A _\star , {_\rhd M_\lhd}) $$ is defined as:

\begin{align*}&\partial^n(f)(a_{1, n+1}) \\
=& (-1)^{n+1}\left(T(a_1)f(a_{2,n+1})-T_M\left( a_1f(a_{2,n+1})\right)  \right) \\
&\sum^{n}_{i=1}(-1)^{n-i+1}\left(  f\left(a_{1,i-1},a_iT(a_{i+1}),a_{i+2,n+1} \right) +f\left( a_{1,i-1},T(a_{i})a_{i+1},a_{i+2,n+1}\right)    +\lambda f\left( a_{1,i-1},a_ia_{i+1},a_{i+2,n+1}\right) \right)\\
&    +  \left(  f\left( a_{1,n}T(a_{n+1})\right) -T_M(f(a_{1,n})a_{n+1})\right)   \\
\end{align*}
for arbitrary $f\in  \C^n_{\Alg}(A_\star ,\  _\rhd M_\lhd)$ and $a_1,\dots,a_{n+1}\in A$.

\begin{prop}\cite{WZ21}
	We can define a chain map $\Phi_{\Alg}^{\bullet}:\C_{\Alg}^\bullet(A ,M)\rightarrow \C^\bullet_{\PLA}(A ,M)$ as follow:
	For  $n\geqslant 1$ and $ f\in \C^n_{\Alg}(A,M)$,  define $\Phi^n(f)\in \C^n_{\RBO_{\lambda}}(A,M)$ as:
$$
		\Phi^n(f) =f\circ (T^{\otimes n})-\sum_{k=1}^{n-1}\lambda^{n-k-1}\sum_{1\leqslant i_1<i_2<\cdots< i_{k}\leqslant n}  T_M\circ f\circ (\Id^{\otimes (i_1-1)} \otimes  T \otimes \Id^{\otimes (i_2-i_1-1)} \otimes T\otimes \cdots\otimes T\otimes \Id^{\otimes (n-1-i_{k})} ) .
$$
\end{prop}
 We define the  cochain complex $(\C^\bullet_{\mathsf{RBA}_{\lambda}}(A,M), d^\bullet)$  of Rota-Baxter pre-Lie algebra $(A,\mu,T)$ with coefficients in $(M,\cdot, T_M)$ to the negative shift of the mapping cone of $\Phi^\bullet$

Let $A$ be a linear space, there exists an inclusion $\Psi^n$ from $A^{\wedge n}$ to $A^{\otimes n}$, $\Psi^n(a_1\wedge\cdots\wedge a_n)=\sum(-1)^{\sigma}a_{\sigma(1)}\otimes\cdots\otimes a_{\sigma(n)}$, for $n\geq 1$ and $a_1,\ldots,a_n\in A$, which can induce a chain map $(\Psi^{*})^{\bullet}:\C_{\Alg}^\bullet(A ,M)\rightarrow \C^\bullet_{\PLA}(A ,M)  $. Thus, we have the following propositon.

\begin{prop}
	Let $A$ and $M$ be two linear spaces. If there exist a Rota-Baxter associative algebra structure $(A,\ T_{\Alg},\ \mu_{\Alg})$ and a Rota-Baxter pre-Lie algebra structure $(A,\ T_{\PLA},\ \mu_{\PLA})$ of weight $\lambda$ on $A$, and a Rota-Baxter	bimodule over $(A,\ T_{\Alg},\ \mu_{\Alg})$ and $(A,\ T_{\PLA},\ \mu_{\PLA})$. Then, we can induce a chain map $(\tilde{\Psi}^{*})^{\bullet}:\C^\bullet_{\mathsf{RBA}_{\lambda}}(A,M)\rightarrow\C^\bullet_{\RBA}(A,M)$.
\end{prop}

\bigskip

\noindent
{{\bf Acknowledgments.} This work is supported in part by Natural Science Foundation of China (Grant Nos. 12071137, 12161013, 11971460).

	

\end{document}